\documentclass{article}

\usepackage{amsmath}
\usepackage{amsfonts}
\usepackage{graphicx, subfigure}
\usepackage[pdftex]{hyperref}
\usepackage[dutch,english]{babel}
\usepackage{amsmath,amssymb}
\usepackage{amsthm}
\usepackage[usenames,dvipsnames]{color}
\usepackage{tikz}
\usepackage{accents}
\usepackage{xcolor}
\usepackage{hyperref}
\usepackage{mathtools}
\usepackage{multicol}
\usepackage{comment}
  \usepackage{color}
\usepackage{hyperref}
  \usetikzlibrary{arrows,shadows,shapes,positioning,shadows,trees}
\usetikzlibrary{decorations.text}
\usetikzlibrary{decorations.pathmorphing}
\usetikzlibrary{calc}
\usepackage{wrapfig}
\usepackage{multirow}
\usepackage{colortbl}

\newcommand{\bl}[1]{{\color{blue}#1}}
\newcommand{\ro}[1]{{\color{red}#1}}

\newcommand{\ora}[1]{{\color{orange}#1}}

\newcommand{\la}{\lambda}

\theoremstyle{plain}
\newtheorem{thr}{Theorem}[section]

\newtheorem{lem}[thr]{Lemma}
\newtheorem{prop}[thr]{Proposition}
\newtheorem{cor}[thr]{Corollary}
\theoremstyle{definition}
\newtheorem{defi}[thr]{Definition}
\newtheorem{ex}[thr]{Example}
\theoremstyle{remark}
\newtheorem{remk}{Remark}
\theoremstyle{remark}

\newcommand{\field}[1]{\mathbb{#1}}
\newcommand{\R}{\field{R}}

\definecolor{wred}{rgb}{0.7,0.18,0.12}
\definecolor{wgreen}{rgb}{0.1,0.53,0.37}

\numberwithin{equation}{section}

\def\barray{\begin{eqnarray*}}             \def\earray{\end{eqnarray*}}
\def\beq{\begin{equation}} \def\eeq{\end{equation}}

\title{Quiver representations and dimension reduction in dynamical systems}
\author{Eddie Nijholt\footnote{\mbox{ Department of Mathematics, University of Illinois, USA, \href{mailto:eddie.nijholt@gmail.com}{eddie.nijholt@gmail.com} }   }, Bob Rink\footnote{ \mbox{Department of Mathematics, Vrije Universiteit Amsterdam, The Netherlands, \href{mailto:b.w.rink@vu.nl}{b.w.rink@vu.nl} }} {} and S\"oren Schwenker\footnote{ \mbox{Department of Mathematics, Universit\"at Hamburg, Germany, \href{mailto:soeren.schwenker@uni-hamburg.de}{soeren.schwenker@uni-hamburg.de} }}}
\date{\today}

\excludecomment{percent}

\begin{document}

\maketitle

\begin{abstract} Dynamical systems often admit geometric properties that must be taken into account when studying their behaviour. We show that many such properties can be encoded by means of quiver representations. These properties include classical symmetry, hidden symmetry and feedforward structure, as well as subnetwork and quotient relations in network dynamical systems. A quiver equivariant dynamical system consists of a collection of dynamical systems with maps between them that send solutions to solutions. We prove that such quiver structures are preserved under Lyapunov-Schmidt reduction, center manifold reduction, and  normal form reduction. 
\end{abstract}

\section{Introduction}\label{sec0}
In this paper we show that various structural properties of dynamical systems (ODEs and iterated maps) can be encoded using the language of quiver representations. These structural properties include classical symmetry, but also feedforward structure, subnetwork and quotient relations in network dynamical systems, and so-called hidden symmetry, including interior symmetry and quotient symmetry. This paper aims to provide a unifying framework for studying dynamical systems with  quiver symmetry. 

A quiver representation consists of a collection of vector spaces with linear maps between them. A simple example is the representation of a group (where there is only one vector space). We shall speak of a dynamical system with quiver symmetry when a dynamical system is defined on each of the vector spaces of the representation, and so that all the linear maps in the representation send the orbits of one dynamical system to orbits of another one.

We will argue that  quiver symmetry is quite prevalent in dynamical systems that have the structure of an interacting network. Essentially, this insight can already be found in the work of Golubitsky, Stewart et al \cite{curious}, \cite{golstew},  \cite{torok}, \cite{stewartnature}, \cite{pivato}, who realised that every trajectory of a so-called {\it quotient} of a network gives rise to a trajectory in the original network system. DeVille and Lerman \cite{deville} later generalised this result and formulated it in the language that we shall use in this paper. Inspired by these ideas, we shall define two distinct quiver representations for each network dynamical system - the quiver of subnetworks and the quiver of quotient networks - and we will    investigate how  these quivers impact the dynamics of a network. 

The advantage of interpreting a  property of a dynamical system as a quiver symmetry, lies in the fact that quiver symmetry is an intrinsic property of a dynamical system. It is for example preserved under composition of maps and Lie brackets of vector fields. 
Unlike for example network structure, which is generally destroyed when a coordinate transformation is applied, quiver symmetry is thus 
defined in a coordinate-invariant manner. This motivates us to start developing a theory for dynamical systems with quiver symmetry. 

As a consequence of its intrinsic definition, quiver symmetry can be incorporated quite easily in many of the  tools that are available for the analysis of dynamical systems.  In this paper, we focus on the impact of quiver symmetry on local dimension reduction techniques. It is well-known that classical symmetry (of compact group actions) is preserved by various of these reduction techniques - see \cite{field3}, \cite{field4}, \cite{fieldgolub}, \cite{constrainedLS}, \cite{perspective}, \cite{golschaef2} and references therein for an overview of results. In this paper, we will generalise these results, by proving that quiver symmetry can be preserved in Lyapunov-Schmidt reduction, center manifold reduction and normal form reduction. More precisely, we show that the dynamical systems that result after applying these  reduction techniques inherit the quiver symmetry of the original dynamical system. Partial results in this direction were obtained by the authors in earlier papers  \cite{fibr}, \cite{CMR}, \cite{CMRSIREV}, \cite{RinkSanders2}, \cite{RinkSanders3}, \cite{CCN}, \cite{schwenker}. 
This paper provides a unifying context for these earlier results. Because the results in this paper apply to any (finite) quiver, we shall not yet try to use any of the more involved results from the theory of quiver representations, such as Gabriel's classification theorem \cite{Gabriel}.

We will start this paper with a  simple  illustrative example of a dynamical system with quiver symmetry in Section \ref{sec2}. We then define quiver representations and quiver equivariant maps in Section \ref{sec3}. In Sections \ref{sec4} and \ref{sec5} we discuss two natural examples of quiver representations that one encounters in the study of network dynamical sytems. In Section \ref{sec6} we gather some properties of endomorphisms of quiver representations. This prepares us to prove the results on Lyapunov-Schmidt reduction, center manifold reduction and normal forms  in Sections \ref{sec7}, \ref{sec8} and \ref{sec9}. We finish the paper with an example in Section \ref{sec10}.

\section{A simple feedforward system}\label{sec2}
Before describing our results in more generality, let us start with a simple example. To this end, let $E_1$ and $E_2$ be finite dimensional real vector spaces and consider a differential equation of the feedforward form 
\begin{align} \label{ffw} 
\left\{\begin{array}{l} 
\frac{dx}{dt} = f(x)\, ,\\
\frac{dy}{dt} = g(x,y)\, , 
\end{array}\right. 
\end{align}
where $x\in E_1$ and $y\in E_2$. We will  show that such a feedforward system can in fact be thought of as a system with quiver symmetry. To explain this, let us (artificially) replace (\ref{ffw}) by two separate systems of differential equations
\begin{align} \label{ffwpairA}
& \left\{ \begin{array}{l} 
\frac{dx}{dt} = f(x)\, , \\
\frac{dy}{dt} = g(x,y)\ , \, 
\end{array} \right.   \hspace{.5cm} \mbox{for} \ (x,y) \in E_1\times E_2\, , \\  \label{ffwpairB}
 & \left\{  \frac{dX}{dt} = f(X)\right. \, , \hspace{.9cm} \mbox{for} \ X \in E_1 \, .
\end{align}
This unconventional step allows us  to formulate the following simple lemma. It states that a feedforward system can be thought of as a system with symmetry.
\begin{lem}\label{fflemma}
A pair of (systems of) differential equations
\begin{align} 
\label{ffwpairgen1}
& \left\{ 
\begin{array}{l} 
\frac{dx}{dt} = F(x, y)\\
\frac{dy}{dt} = G(x,y)\, 
\end{array} 
\right. 
\\ \label{ffwpairgen2}
&  \left\{  \frac{dX}{dt} = H(X)\right.
\end{align}
is of the feedforward form  (\ref{ffwpairA}),  (\ref{ffwpairB}) if and only if the map 
$$R: E_1\times E_2 \to E_1 \ \mbox{defined\ by}\ R(x,y):=x$$ 
sends every solution of (\ref{ffwpairgen1})  to a solution of (\ref{ffwpairgen2}).
\end{lem}
\begin{proof}
Firstly, assume that (\ref{ffwpairgen1}) and (\ref{ffwpairgen2}) are actually of the form (\ref{ffwpairA}), (\ref{ffwpairB}) respectively. This means that   
$$F(x,y)=f(x), G(x,y)=g(x,y), H(X)=f(X)\, .$$ 
Assume now  that $(x(t), y(t))$ solves (\ref{ffwpairgen1}). Then $\frac{dx}{dt}(t) = F(x(t),y(t))$ and hence $X(t):=R(x(t), y(t)) = x(t)$ satisfies 
$$\frac{dX}{dt}(t) = \frac{dx}{dt}(t) = F(x(t),y(t)) =f(x(t)) = f(X(t))=H(X(t))\, .$$ 
So $R$ sends solutions of (\ref{ffwpairgen1}) to solutions of (\ref{ffwpairgen2}).

For the other direction, assume that for every solution $(x(t), y(t))$ of (\ref{ffwpairgen1}) the curve $X(t):=R(x(t), y(t)) = x(t)$ is a solution of (\ref{ffwpairgen2}). Let $(x, y) \in E_1\times E_2$ be arbitrary, and let $(x(t), y(t))$ be the solution of (\ref{ffwpairgen1}) with $(x(0), y(0))=(x,y)$. Define $X := R(x,y)=x$ and $X(t):=R(x(t), y(t))= x(t)$. Then $X(0)=X$ and $\frac{dX}{dt}(t) = H(X(t))$. It follows that $F(x,y) = \frac{dx}{dt}(0) = \frac{dX}{dt}(0) = H(X(0))=H(X) = H(x)$. So $F(x,y)=H(x)$ for all $x,y$. If we now define $f(X):=H(X)$ and $g(x,y) := G(x,y)$, then obviously $F(x,y) = H(x) = f(x)$. In other words, (\ref{ffwpairA})  coincides with (\ref{ffwpairgen1})  and  (\ref{ffwpairB}) coincides with  (\ref{ffwpairgen2}).
 \end{proof}
\noindent Lemma \ref{fflemma} translates the property that an ODE has feedforward structure into a somewhat unconventional symmetry property. We will see many more examples of this phenomenon later. 
It is important to note that the symmetry in Lemma \ref{fflemma} is a noninvertible map between two different vector spaces. We are thus not in the classical setting where the symmetries form a group. Instead, they form a (rather simple) quiver. 
 
 The statement of the following  lemma is not new, see \cite{curious}, but we provide a new proof that is based on the observation in Lemma \ref{fflemma}. This proof nicely illustrates how quiver symmetry can be taken into account when we analyse a dynamical system. Moreover, the proof below easily generalises to dynamical systems with more complicated quiver symmetries, see Theorem \ref{CMtheorem} below. 
\begin{lem}\label{cmff}
Let $(x_0,y_0)$ be an equilibrium point of the feedforward system 
\begin{align} \label{ffwpairAac}
& \left\{ \begin{array}{l} 
\frac{dx}{dt} = f(x)\, ,\\ 
\frac{dy}{dt} = g(x,y)\, ,\, 
\end{array} \right.    \hspace{.5cm} \mbox{for} \ (x,y) \in E_1\times E_2\, .
\end{align}
Denote by $L_{(x_0, y_0)}$ the  Jacobian of (\ref{ffwpairAac}) at $(x_0, y_0)$ and by $E_1\times E_2 = E^c \oplus E^h$  the decomposition into its  center and hyperbolic subspaces.  We denote by $$\pi^c: E_1\times E_2 = E^c \oplus E^h \to E^c \  , \ (x, y)\mapsto (x^c, y^c)$$  the projection onto $E^c$ along $E^h$. 
Assume that (\ref{ffwpairAac}) admits a global center manifold at $(x_0, y_0)$. Then $\pi^c$ conjugates the dynamics on this center manifold to a dynamical system on $E^c$ of the form
\begin{align} \label{ffwpairAacred}
& \left\{ \begin{array}{l} 
\frac{d x^c}{dt} = f^c(x^c)\, ,\\ 
\frac{d y^c}{dt} = g^c(x^c, y^c)\, . 
\end{array} \right. 
\end{align}
 
 
\end{lem}
\begin{proof}
Let us define $F: E_1\times E_2 \to E_1 \times E_2$ by $F(x,y):=(f(x), g(x,y))$. Then equation (\ref{ffwpairAac}) can be written as $\frac{d}{dt}(x,y)=F(x,y)$. Recall from Lemma \ref{fflemma} that the feedforward structure of $F$ implies that $R: E_1\times E_2 \to E_1, (x,y)\mapsto x$ sends solutions of this ODE to solutions of the ODE
\begin{align}\label{smaller}
\frac{dX}{dt}=f(X)\, .
\end{align}
In other words, we have that 
$$R \circ F = f \circ R\, .$$
This clearly implies that $X_0:=R(x_0,y_0)=x_0$ is an equilibrium point of  (\ref{smaller}), and  if we write $L_{X_0}=Df(X_0)$ for the Jacobian of (\ref{smaller}) at $X_0$, then 
\begin{align}\label{linearconjugacy}
R \circ L_{(x_0, y_0)}   = L_{X_0} \circ R\, .
\end{align}
This follows from differentiating $R(F(x,y)) = f(R(x,y))$ at $(x,y)=(x_0, y_0)$. Let us decompose $E_1 = E^{c'}\oplus E^{h'}$ into the center and hyperbolic subspaces of $L_{X_0}$. Then it follows from (\ref{linearconjugacy}) that $R$ maps any generalised eigenspace of $L_{(x_0, y_0)}$ into the generalised eigenspace of $L_{X_0}$ with the same eigenvalue. It follows that  $R(E^c)\subset E^{c'}$ and $R(E^h)\subset E^{h'}$. Denoting by $\Pi^c: E_1 = E^{c'}\oplus E^{h'} \to E^{c'}$ the projection onto $E^{c'}$ along $E^{h'}$, we conclude that
$$\Pi^c \circ R = R \circ \pi^c \, .$$
The next step is to prove that $R$ sends the global center manifold of (\ref{ffwpairAac}) at $(x_0, y_0)$ to the global center manifold of (\ref{smaller}) at $X_0$.  
For this we recall that a solution $(x(t), y(t))$ lies in the global center manifold of (\ref{ffwpairAac}) if and only if 
$$\sup_{t\in \mathbb{R}} || \pi^h(x(t), y(t))|| < \infty\, ,$$
where we write $\pi^h = 1-\pi^c$. We similarly write $\Pi^h=1-\Pi^c$. If we define $X(t) := R(x(t), y(t)) = x(t)$ for such a solution, then clearly
$\Pi^h(X(t)) = \Pi^h(R(x(t), y(t)) = R (\pi^h(x(t),y(t)))$, and so
$$\sup_{t\in \mathbb{R}} || \Pi^h(X(t))|| \leq ||R||  \cdot \sup_{t\in \mathbb{R}}||  \pi^h(x(t), y(t))|| < \infty \, ,$$ 
where  $||R||$ denotes the operator norm of $R$ (which equals $1$ here). So $X(t)$ lies in the global center manifold of (\ref{smaller}). This proves that $R$ maps the center manifold of (\ref{ffwpairAac}) into the center manifold of (\ref{smaller}). 

Next, recall that the center manifolds of  (\ref{ffwpairAac})  and (\ref{smaller})  are the graphs of certain (finitely many times continuously differentiable) functions $\phi: E^c \to E^h$ and $\psi: E^{c'}\to E^{h'}$ respectively. In other words, for every $(x,y)$  in the center manifold of (\ref{ffwpairAac}) and $X$  in the center manifold of (\ref{smaller}), we have
\begin{align} \label{cmgraph}
 (x, y) & = \underbrace{(x^c, y^c)}_{\in E^c} +  \underbrace{\phi(x^c, y^c)}_{\in E^h} \, ,
\\ \label{cmgraph2}
  X &  = \underbrace{X^c}_{\in E^{c'}} + \underbrace{\psi(X^c)}_{\in E^{h'}}\, .
\end{align}
Pick an $(x,y)$ in the center manifold of (\ref{ffwpairAac}). 
Applying $R$ to (\ref{cmgraph})  yields that 
$$x = x^c + R(\phi(x^c, y^c))\, .$$ 
Note that this $x$  lies in the center manifold of (\ref{smaller}) by the result above. We also have that $x^c = R(x^c, y^c) \in E^{c'}$ because $(x^c, y^c) \in E^c$ and $R(E^c)\subset E^{c'}$. Similarly, $R(\phi(x^c, y^c))\in E^{h'}$ because $\phi(x^c, y^c) \in E^h$ and $R(E^h)\subset E^{h'}$.  But this means  that $x^c$ is the center part of $x$ and $R(\phi(x^c, x^h))$ is its hyperbolic part. So  (\ref{cmgraph2}) gives that $R(\phi(x^c, y^c))$ must be equal to $\psi(x^c)=\psi(R(x^c, y^c))$.
This proves that 
 $$R \circ \phi = \psi \circ R \, .$$
Next, let $(x(t), y(t))$ be an integral curve of $F$ lying inside the center manifold of (\ref{ffwpairAac}), and let us once again write 
$$(x(t), y(t)) =  (x^c(t), y^c(t))  +  \phi(x^c(t), y^c(t))\, .$$
Because  $\frac{d}{dt}(x(t),y(t)) = F(x(t), y(t))$, it then follows that
$$\frac{d}{dt}(x^c(t), y^c(t))  = (\pi^c\circ F) ( (x^c(t), y^c(t))+ \phi(x^c(t), y^c(t)))\, .$$
This shows that the restriction of $\pi^c$ to the center manifold sends integral curves of $F$ to integral curves of the vector field $F^c: E^{ c} \to E^{ c}$ defined by 
\begin{align}
F^c(x^c, y^c) & := (\pi^c \circ F)((x^c, y^c) + \phi( x^c, y^c))\, . \nonumber  
\end{align}
Similarly, the restriction of $\Pi^c$ to the center manifold of (\ref{smaller}) sends integral curves of $f$ to integral curves of $f^c: E^{c'}\to E^{c'}$ defined by
 \begin{align}\nonumber
f^c(X^c)  & := (\Pi^c \circ f)(X^c + \psi(X^c))\, .
\end{align}
Now we simply notice that 
\begin{align}\nonumber 
R(F^c(x^c, y^c)) =&  (R\circ \pi^c \circ F)((x^c, y^c) + \phi( x^c, y^c))  \\ \nonumber
 = & (\Pi^c \circ R \circ F)((x^c, y^c) + \phi( x^c, y^c))  \\ \nonumber
= & (\Pi^c \circ f \circ R)((x^c, y^c) + \phi( x^c, y^c))  \\ \nonumber
= & (\Pi^c \circ f) ( R (x^c, y^c) + R( \phi( x^c, y^c)))  \\ \nonumber
= & (\Pi^c \circ f) ( R(x^c, y^c) + \psi( R (x^c, y^c)))  \\ \nonumber
= & f^c(R(x^c, y^c)))\, ,
\end{align}
which proves that $$R\circ F^c = f^c \circ R\, . $$
In this last formula, $R$ in fact denotes the restriction $R:E^c \to E^{c'}$ given by $R(x^c, y^c)=x^c$. 
Lemma \ref{ffw} thus guarantees that $F^c$ is of the feedforward form $$F^c(x^c, y^c)=(f^c(x^c), g^c(x^c, y^c))$$ for some function $g^c:E^c\to E^c$. This finishes the proof. 
\end{proof}

\section{Quiver equivariant dynamical systems}\label{sec3}
The pair of ODEs (\ref{ffwpairA}), (\ref{ffwpairB}) is a simple example of a quiver equivariant dynamical system. We shall now give the general definition. In this paper, we only consider quivers with finitely many vertices and arrows, because this simplifies our proofs.
\begin{defi} \mbox{}\\ \vspace{-5mm}
\begin{itemize} 
\item[{\it i)}] A {\it quiver} is a directed (multi)graph 
$${\bf Q} = \{A\rightrightarrows^s_t V\}$$ 
 consisting of a finite set of arrows $A$, a finite set of vertices $V$, a source map $s: A \to V$ and a target map $t: A\to V$. 
\item[{\it ii)}] A {\it representation} ({\bf E}, {\bf R}) of a quiver {\bf Q} consists of a set {\bf E} of finite dimensional vector spaces $E_v$ (one for each vertex $v\in V$), and a set {\bf R} of linear maps 
$$R_{a}: E_{s(a)} \to E_{t(a)}\ \  \mbox{(one for each arrow}\ a\in A).$$
\item[{\it iii)}] A {\bf Q}-{\it equivariant map} {\bf F} of a representation ({\bf E}, {\bf R}) of a quiver {\bf Q} consists of a collection of maps $F_v: E_v\to E_v$ (one for each vertex $v\in V$) so that 
$$ F_{t(a)} \circ R_a = R_a \circ F_{s(a)} \ \ \mbox{for every arrow}\ a\in A\, .$$
We shall write ${\bf F}\in C^{\infty}({\bf E}, {\bf R})$ if $F_v \in C^{\infty}(E_v)$ for every $v\in V$. We shall sometimes refer to a {\bf Q}-equivariant map as a {\bf Q}-{\it equivariant vector field} or  a {\bf Q}-{\it equivariant dynamical system}.
\end{itemize}
\end{defi}
\noindent The following simple proposition expresses that quiver-equivariance is an intrinsic property.  Proposition \ref{Lieproperty} formulates the corresponding result for the Lie bracket.
\begin{prop}\label{obviouscomp}
Let $({\bf E}, {\bf R})$ be a representation of a quiver ${\bf Q}= \{A\rightrightarrows^s_t V\}$ and let ${\bf F}, {\bf G} \in C^{\infty}({\bf E}, {\bf R})$. Define the composition ${\bf F} \circ {\bf G}$ to consist of the maps $(F_v \circ G_v): E_v \to E_v$ (for $v\in V$). Then ${\bf F} \circ {\bf G} \in C^{\infty}({\bf E}, {\bf R})$.
\end{prop}
\begin{proof}
Smoothness of $({  F} _v\circ {  G}_v)$ is obvious. Now let $a\in A$ be an arrow. Then 
\begin{align}\nonumber
 R_a \circ (F_{s(a)} \circ G_{s(a)}) =   F_{t(a)} \circ R_a \circ G_{s(a)}  = (F_{t(a)} \circ  G_{t(a)}) \circ R_a \, .
\end{align}
\end{proof}
\noindent The next example shows that the feedforward system of Section \ref{sec2} constitutes a quiver equivariant dynamical system.
\begin{ex} \label{ffexamplequiver}
Consider a quiver {\bf Q} consisting of two vertices $V=\{v_1, v_2\}$ and three arrows $A=\{a_1, a_2, a_3\}$, where $s(a_1) = t(a_1) = v_1$ and $s(a_2)=v_1$ and $t(a_2)=v_2$ and $s(a_3) = t(a_3) = v_2$. Define
$$E_{v_1} = E_1\times E_2 \ \mbox{and} \ E_{v_2} = E_1$$ 
with $E_1$ and $E_2$ vector spaces, and
\begin{align} \nonumber 
  &R_{a_1}(x,y) = (x,y)\, , \\  \nonumber 
   &R_{a_2}(x,y) = x\, , \\ \nonumber  
     &R_{a_3}(X) = X\, .\ 
 \end{align}
Then the pair of  maps 
\begin{align} \nonumber& F_{v_1} = (F, G): E_{v_1} = E_1\times E_2 \to E_{v_1} = E_1\times E_2 , \\ \nonumber
&F_{v_2} = H: E_{v_2} = E_1 \to E_{v_2} = E_1
\end{align} 
is  {\bf Q}-equivariant if and only if $R_{a_2} \circ F_{v_1} = F_{v_2}\circ R_{a_2}$, that is, if
\begin{align} \nonumber
H(x) = & F_{v_2} (x) = F_{v_2} (R_{a_2}(x, y)) \\ \nonumber = & R_{a_2}  (F_{v_1}(x,y)) = R_{a_2}  (F(x,y), G(x,y)) = F(x,y)\, .
\end{align}
So {\bf Q}-equivariance just means that $F_{v_1}$ is of feedforward form.
\end{ex}

\begin{ex}
Let $E$ be a representation of a finite group $G$. This means that $E$ is a vector space and that for every $g\in G$ there is a (necessarily invertible) linear map $R_g: E\to E$, such that $R_e={\rm Id}_E$ and $R_{g_1} \circ R_{g_2} = R_{g_1g_2}$. 

Such a group representation can be thought of as a representation of a quiver with one vertex, say $V=\{v\}$, and exactly one arrow $a=a(g)$ (to and from $v$) for each $g\in G$.  This is done by defining $E_v:=E$ and $R_{a(g)} :=R_g$.  The quiver equivariant maps are then simply the maps $F:E\to E$ with 
$$F \circ R_{g}  = R_g \circ F \ \mbox{for all} \ g\in G\ .$$
So the quiver equivariant maps coincide with the usual $G$-equivariant maps. 
\end{ex}
\begin{ex}
As a straightforward generalisation of the previous example, one may study linear maps that are not invertible.  
Consider for example the map 
$$R: x\mapsto 0\ \mbox{from}\ \mathbb{R}\ \mbox{to}\ \mathbb{R}\, .$$
This map defines a representation of a quiver with just one vertex $v\in V$ and one arrow $a\in A$, where $E_v:=\mathbb{R}$ and $R_a:=R$.  

Note that an ODE $\frac{dx}{dt} = F(x)$ satisfies $F\circ R= R \circ F$ if and only if $F(0)=0$. So having this quiver symmetry is equivalent to having a steady state at the origin. 
Interestingly, this is the setting in which the  {\it transcritical bifurcation}
$$\frac{dx}{dt} = \lambda x \pm x^2$$
is the typical one-parameter bifurcation. Curiously, this shows that the transcritical bifurcation is  a generic quiver equivariant bifurcation.
\end{ex}

\begin{ex} \label{monoidex}
In \cite{fibr} it turned out natural to study dynamical systems that are equivariant under the action of a finite monoid $\Sigma$. A monoid is a set $\Sigma$ with an associative multiplication $(\sigma_1, \sigma_2)\mapsto \sigma_1 \sigma_2$ and a multiplicative unit $\sigma_0$. A representation of $\Sigma$ consists of  (not necessarily invertible)  linear maps $R_{\sigma}: E\to E$ on a vector space $E$, so that $R_{\sigma_0}={\rm Id}_{E}$ and $R_{\sigma_1} \circ R_{\sigma_2} = R_{\sigma_1\sigma_2}$. 

This setup arises for example when studying the network in Figure \ref{pictfundamental}. The figure displays a network with five nodes and a map $F = F(x_1, x_2, x_3, x_4, x_5)$ that is ``compatible'' with the structure of this network.

\vspace{-5mm}
\begin{figure}[h]\renewcommand{\figurename}{\rm \bf \footnotesize Figure} 
\begin{center}
\hspace{.4cm}
\begin{tabular}{p{4cm}p{4.5cm}} \hspace{.3cm}
 \begin{tikzpicture}[->, scale=1.9]
	 \tikzstyle{vertextype2} = [circle, draw, fill=yellow, minimum size=15pt,inner sep=1pt]
	 \node at (1,1.4) {};
	 \node at (1,-.3) {};
	 \node[vertextype2] (v1) at (1.85,0) {$1$};
	 \node[vertextype2] (v2) at (.15,0) {$2$};
	 \node[vertextype2] (v3) at (1.5,.35) {$3$};
	  \node[vertextype2] (v4) at (1,1.05) {$4$};
	 \node[vertextype2] (v5) at (.5,.35) {$5$};
	\path[]
	(v2) edge [bend right, blue, thick] node {} (v1)
	(v4) edge [thick, bend right, blue] node {} (v2)
	(v5) edge [thick, blue] node {} (v3)
	(v4) edge [loop left, thick, blue] node {} (v4)
	(v4) edge [thick, blue] node {} (v5)
	(v3) edge [red, thick, dashed] node {} (v1)
	(v3) edge [red, thick, bend left, dashed] node {} (v2)
	(v3) edge [loop right, red, thick, dashed] node {} (v3)
	(v3) edge [red, thick, dashed] node {} (v4)
	(v3) edge [bend left, red, thick, dashed] node {} (v5);
 \end{tikzpicture} 
 &\vspace{.2cm} 
 \begin{equation}   \nonumber 
F\left( \begin{array}{c} 
x_1 \\ x_2 \\ x_3 \\ x_4 \\ x_5
\end{array} 
\right) = \left(
 \begin{array}{l}   
    f(x_{1}, \bl{ x_{2}}, \ro{x_{3}}) \\
 f(x_{2}, \bl{x_{4}}, \ro{x_{3}}) \\
  f(x_{3}, \bl{x_{5}}, \ro{x_{3}})\\ 
  f(x_{4}, \bl{x_{4}}, \ro{x_{3}})\\
 f(x_{5}, \bl{x_{4}}, \ro{x_{3}})
 \end{array}\right) 
 \hspace{-.8cm}
 \end{equation} 
\end{tabular} 
\vspace{-3mm}
 \caption{\footnotesize {\rm A network map with a monoid of $5$ symmetries.}}
 \vspace{-.5cm}
\label{pictfundamental}
\end{center}
 \end{figure}
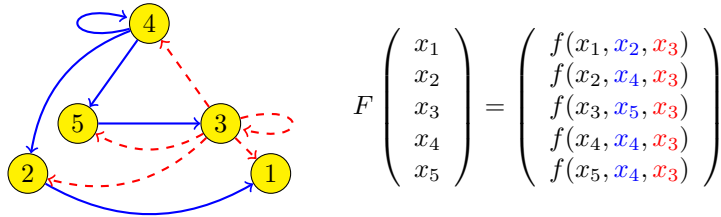
It turns out that  an $F$ of this form always commutes with the maps 
\begin{align}\nonumber
& R_{\sigma_0}(x_1, x_2, x_3, x_4, x_5) =  (x_1, x_2, x_3, x_4, x_5) \, ,\\ \nonumber
& R_{\sigma_1}(x_1, x_2, x_3, x_4, x_5) =  (x_2, x_4, x_3, x_4, x_5)\, , \\ \nonumber
& R_{\sigma_2}(x_1, x_2, x_3, x_4, x_5) = (x_3, x_5, x_3, x_4, x_5) \, , \\ \nonumber
& R_{\sigma_3}(x_1, x_2, x_3, x_4, x_5) =  (x_4, x_4, x_3, x_4, x_5)\, , \\ \nonumber
& R_{\sigma_4}(x_1, x_2, x_3, x_4, x_5) = (x_5, x_4, x_3, x_4, x_5)\, .
\end{align}
These maps together form a representation of a monoid $\Sigma$ with $5$ elements.  In \cite{fibr} this representation was used to classify the bifurcations that occur in the dynamics of the ODE $\frac{dx}{dt} = F(x)$. We will not discuss these results in any detail here.
Just like for groups, one may think of a representation  of a monoid as a special case of a representation of a quiver.
\end{ex}
\begin{remk}
The notion of {\it interior network symmetry} was defined in \cite{pivato2}. We will not discuss interior symmetry in any detail here, but we would like to point out that interior symmetry is equivalent to a special type of quiver symmetry, for a quiver with two vertices and a possibly quite large number of arrows. This fact was proved in Section 9 of \cite{fibr}. 
\end{remk}
\noindent In the coming sections we provide more examples of dynamical systems with quiver-symmetry. We start by generalising Example \ref{ffexamplequiver} to include more general network systems. 

\section{The quiver of subnetworks}\label{sec4}

In this section and  the next we consider dynamical systems with the structure of an interacting network. We apologise for the somewhat heavy notation in this section, which we found impossible to avoid.  

We start by letting ${\bf N} = \{ E \rightrightarrows^s_t N\}$ be a directed graph consisting of a finite number of nodes $n \in N$ and directed edges $e \in E$ (the letter {\bf N} stands for {\it network}). This ${\bf N}$ should not be thought of as a quiver (we shall use ${\bf N}$ to define a quiver ${\bf SubQ}({\bf N})$ later) but as the network structure of an iterated map or ODE. More precisely, we assume that for each   $n \in N$ we are given a vector space  $E_n$ (the so-called ``internal phase space'' of this node) and a map 
$$F_{n}: \bigoplus_{e \in E \, :\,  t(e)  = n} E_{s(e)} \to E_{n}\, .$$
So $F_{n}$ depends only on those $x_{m}$ for which there is an edge $e$ from ${m}$ to ${n}$. 
Together the $F_n$ define a {\it network map}
$F^{\bf N}: \bigoplus_{{  m} \in {  N}} E_{  m}  \to  \bigoplus_{{  m} \in {  N}} E_{  m}$
given by 
$$F^{\bf N}_n\left( \, \bigoplus_{{  m} \in N} x_m\, \right) = F_{  n}\left(  \bigoplus_{{ e} \in { E}\, : \, t(e) = n } x_{ s(e)}   \right) \, .$$
One could say that this $F^{\bf N}$ is ``compatible'' with the network  ${\bf N}$. We may  use $F^{\bf N}$ to define a ``network dynamical system'' on the ``total phase space'' $\bigoplus_{{  m} \in {  N}} E_{  m}$, for example the iteration $x^{(n+1)} = F^{\bf N}(x^{(n)})$ or the flow of the ODE $\frac{dx}{dt} = F^{\bf N}(x)$. 

\begin{ex}\label{exffsimple}
The network {\bf N} in Figure \ref{pict2} consist of $2$ nodes (labeled $1$ and $2$) and $3$ arrows.  The network maps compatible with this network are the maps of the form 
$$F^{\bf N}(x_1, x_2) = (F_1(x_1), F_2(x_1, x_2))\, .$$
These are precisely the feedforward maps  of Example \ref{ffexamplequiver}.
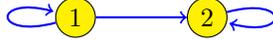
\begin{figure}[h]\renewcommand{\figurename}{\rm \bf \footnotesize Figure} 
\begin{center}
\hspace{0cm}
  \hspace{-.4cm} 
 \begin{tikzpicture}[->, scale=1.75]
	 \tikzstyle{vertextype2} = [circle, draw, fill=yellow, minimum size=15pt,inner sep=1pt]
	 \node[vertextype2] (v1) at (0,0) {$1$};
	 \node[vertextype2] (v2) at (1,0) {$2$};
	\path[]
	(v1) edge [loop left, blue, thick] node {} (v1)
	(v2) edge [loop right, blue, thick] node {} (v2)
	(v1) edge [blue, thick] node {} (v2);   
 \end{tikzpicture} 
 \vspace{-.3cm}
 \caption{\footnotesize {\rm A feedforward network with two nodes.}}
 \vspace{-.6cm}
\label{pict2}
\end{center}
 \end{figure}
\end{ex}
\begin{ex}\label{interestingffex}
Let ${\bf N}$ be the network consisting of $5$ nodes (labeled $1, \ldots, 5$) and $12$ arrows as defined in Figure \ref{pict1}. 
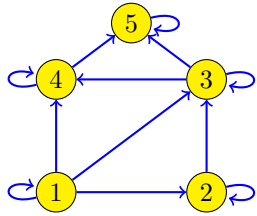
\begin{figure}[h]\renewcommand{\figurename}{\rm \bf \footnotesize Figure} 
\begin{center}
\hspace{0cm}
  \hspace{-.4cm} 
 \begin{tikzpicture}[->, scale=1]
	 \tikzstyle{vertextype1}=[circle, draw, minimum size=20pt,inner sep=1pt]
	 \tikzstyle{vertextype2} = [circle, draw, fill=yellow, minimum size=15pt,inner sep=1pt]
	 \tikzstyle{vertextype3} = [circle, draw, fill=yellow!25, minimum size=20pt,inner sep=1pt]
	 \tikzstyle{edgetype1} = [->, draw,line width=3pt,-,red!50]
	 \tikzstyle{edgetype2} = [draw,thick,-,black]
	 \node[vertextype2] (v1) at (0,0) {$1$};
	 \node[vertextype2] (v2) at (2,1.5) {$3$};
	 \node[vertextype2] (v3) at (0,1.5) {$4$};
	  \node[vertextype2] (v4) at (2,0) {$2$};
	 \node[vertextype2] (v5) at (1,2.25) {$5$};
	\path[]
	(v1) edge [loop left, blue, thick] node {} (v1)
	(v2) edge [loop right, blue, thick] node {} (v2)
	(v3) edge [loop left, blue, thick] node {} (v3)
	(v4) edge [loop right, blue, thick] node {} (v4)
	(v5) edge [loop right, blue, thick] node {} (v5)
	(v1) edge [thick, blue] node {} (v2)
	(v1) edge [thick, blue] node {} (v4)
	(v1) edge [thick, blue] node {} (v3)
	(v2) edge [thick, blue] node {} (v5)
	(v4) edge [thick, blue] node {} (v2)
	(v2) edge [thick, blue] node {} (v3)
	(v3) edge [blue, thick] node {} (v5);
 \end{tikzpicture} 
 \vspace{-.3cm}
 \caption{\footnotesize {\rm A feedforward type network with five nodes.}}
 \vspace{-.4cm}
\label{pict1}
\end{center}
 \end{figure} \\
\noindent Then any network map $F^{\bf N}$ takes the form 
\begin{align}\label{networkform}
F^{\bf N}\left( \begin{array}{l} x_1\\ x_2 \\ x_3 \\ x_4 \\ x_5 \end{array} \right) = 
\left( 
\begin{array}{l} 
  F_1(x_1) \\
  F_2(x_1, x_2) \\
  F_3(x_1, x_2, x_3) \\
 F_4(x_1, x_3, x_4) \\
  F_5(x_3, x_4, x_5) \\
\end{array}
\right)\, .
\end{align}
Thus $F^{\bf N}$ has a rather particular feedforward structure. Note also that when $F^{\bf N}$ and $G^{\bf N}$ are two such network maps, then their composition will have the form
\begin{align}
(F^{\bf N}\circ G^{\bf N}) \! \left( \!\!\! \begin{array}{l} x_1\\ x_2 \\ x_3 \\ x_4 \\ x_5 \end{array} \!\!\! \right) \! = \!  
\left(  \!\!\!
\begin{array}{l} \nonumber
  F_1(G_1(x_1)) \\
  F_2(G_1(x_1), G_2(x_1, x_2)) \\
  F_3(G_1(x_1), G_2(x_1,x_2), G_3(x_1, x_2, x_3)) \\
 F_4(G_1(x_1), G_3(x_1, x_2, x_3), G_4(x_1, x_3, x_4)) \\
  F_5(G_3(x_1, x_2, x_3), G_4(x_1, x_3, x_4), G_5(x_3, x_4, x_5) ) \\
\end{array} \!\!\!
\right) .
\end{align}
This shows that $(F^{\bf N}\circ G^{\bf N})_4(x)$ depends explicitly on $x_2$, while $F_4(x)$ and $G_4(x)$ do not. Similarly, $(F^{\bf N}\circ G^{\bf N})_5(x)$ depends explicitly on $x_1, x_2$, while $F_5(x)$ and $G_5(x)$ do not. So we see that the network structure of $F^{\bf N}$ and $G^{\bf N}$ is destroyed when we compose them. On the other hand, we also observe that a large part of the network structure of $F^{\bf N}$ and $G^{\bf N}$ remains intact in  $F^{\bf N}\circ G^{\bf N}$.
\end{ex}
\noindent In the remainder of this section we will show that network maps admit a specific quiver symmetry. This will clarify which characteristics of the network structure will survive if we, for example, compose network maps. We start with the definition of a subnetwork.
\begin{defi}
Let ${\bf N} = \{ E \rightrightarrows^s_t N\}$ be a network and let $N'\subseteq N$. Assume that for every $e\in E$ with $t(e)\in N'$ it holds that $s(e)\in N'$. Define 
$$E' := \{e\in E \, : \, s(e), t(e) \in N'\}\, .$$     
Then ${\bf N}' = \{E' \rightrightarrows^s_t N' \}$ is called a {\it subnetwork} of ${\bf N}$. We shall write ${\bf N'} \sqsubseteq {\bf N}$. 
\end{defi}
\begin{remk}
\noindent The relation $\sqsubseteq $ defines a partial order on the set of subnetworks of ${\bf N}$. Indeed, ${\bf N}' \sqsubseteq {\bf N}'$ for all ${\bf N}' \sqsubseteq {\bf N}$ (reflexivity), ${\bf N}' \sqsubseteq {\bf N}''$ and ${\bf N}'' \sqsubseteq {\bf N}'$ imply ${\bf N}' = {\bf N}''$ (antisymmetry) and ${\bf N}''' \sqsubseteq {\bf N}''$ and ${\bf N}'' \sqsubseteq {\bf N}'$ together imply that ${\bf N}''' \sqsubseteq {\bf N}'$ (transitivity).
\end{remk}
\noindent We shall use the subnetworks of {\bf N} to define a quiver  as follows. 
\begin{defi}
Let ${\bf N}$ be a network. The quiver ${\bf SubQ}({\bf N}) = \{A \rightrightarrows^s_t V\}$ of subnetworks of ${\bf N}$ has as its  vertices the nonempty subnetworks of ${\bf N}$, i.e.,
$$V=\{ {\bf N}'  \, | \, \emptyset \neq {\bf N}' \sqsubseteq  {\bf N}\, \} \, .$$ 
There is exactly one arrow $a\in A$  with $s(a)={\bf N}'$ and $t(a)={\bf N}''$  if ${\bf N}'' \sqsubseteq {\bf N}'$.
 \end{defi}
\noindent A representation of ${\bf SubQ}({\bf N})$ can  be constructed in the following straightforward manner. Recall that for every $n\in {N}$  there is a vector space $E_n$. We now set
 $$E_{{\bf N}'} := \bigoplus_{m \in N'} E_{m} $$
 and we define, for the arrow $a \in A$ from ${\bf N}'$ to ${ \bf N}''$ (so assuming that $N'' \subseteq N'$),
 \begin{align}\label{Radef}
 R_{a}:  E_{{\bf N}'}   \to E_{{\bf  N}''}   \ \mbox{by} \ R_a\left(\bigoplus_{m \in N'} x_m \right) : = \bigoplus_{m\in N''} x_m \, . 
 \end{align}
 So $R_a$ ``forgets'' the states $x_m$ with $m\in N' \backslash N''$. Before we continue to explain why these definitions are useful, let us briefly return to our two examples.
  \begin{ex}
 Let ${\bf N}$ be the network of Example \ref{exffsimple} and Figure \ref{pict2}. It has two nonempty subnetworks, which we call ${\bf N}_1$ and ${\bf N}_2$. Figure \ref{pict3} depicts the quiver ${\bf SubQ}({\bf N})$. The arrows in the quiver that express the subnetwork relations ${\bf N}_1 \sqsubseteq {\bf N}_1, {\bf N}_1 \sqsubseteq {\bf N}_2$ and ${\bf N}_2 \sqsubseteq {\bf N}_2$ are drawn as snaking arrows. 
 \begin{figure}[h]\renewcommand{\figurename}{\rm \bf \footnotesize Figure} 
 \begin{center}
\hspace{0cm}
  \hspace{-.4cm} 
 \begin{tikzpicture}[->, scale=1.75]
	 \tikzstyle{vertextype2} = [circle, draw, fill=yellow, minimum size=15pt,inner sep=1pt]
	 \node[vertextype2] (v1) at (0,0) {$1$};
	 \node[vertextype2] (v2) at (1,0) {$2$};
	  \node[vertextype2] (v3) at (5,0) {$1$};
	  \node at (.5,-.4) {\Large ${\bf N}_2$};
	   \node at (5,-.4) {\Large ${\bf N}_1$};
	\path[]
	(2,0) edge [thick, decorate, decoration=snake, gray]  node [above, midway] {$a_2\ $} (4,0) 
	(0,.3) edge [thick, decorate, decoration=snake, bend left=110, gray] node [below, midway] {$a_1$} (1,.3)
	(4.5,.3) edge [thick, decorate, decoration=snake, bend left=110, gray] node [below, midway] {$a_3$} (5.5,.3)
	(v1) edge [loop left, blue, thick] node {} (v1)
	(v2) edge [loop right, blue, thick] node {} (v2)
	(v1) edge [blue, thick] node {} (v2)
	(v3) edge [blue, thick, loop left] node {} (v3);   
 \end{tikzpicture} 
 \vspace{-.3cm}
 \caption{\footnotesize {\rm  Subnetwork quiver for the feedforward network in Figure \ref{pict2}.}}
 \vspace{-.3cm}
\label{pict3}
\end{center}
 \end{figure}
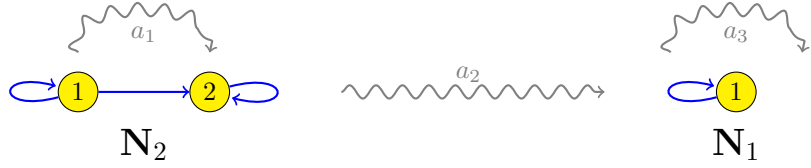
  It should be clear that the linear maps defining the representation are given by $R_{a_1}(x_1, x_2) = (x_1, x_2)$, $R_{a_2}(x_1, x_2) = x_1$ and $R_{a_3}(x_1) = x_1$. 
 \end{ex}
 
 \begin{ex}\label{ex2quiver}
 Let {\bf N} be the network of Example \ref{interestingffex} as depicted in Figure \ref{pict1}. It has five nonempty subnetworks, which we depict in Figure \ref{pict4}. The figure also depicts some (but not all) of the arrows in the subnetwork quiver. 
 \begin{figure}[h]\renewcommand{\figurename}{\rm \bf \footnotesize Figure} 
\begin{center}
\hspace{0cm}
  \hspace{-.4cm} 
 \begin{tikzpicture}[->, scale=.75]
	 	 \tikzstyle{vertextype2} = [circle, draw, fill=yellow, minimum size=10pt,inner sep=1pt]

	 \node[vertextype2] (v1) at (-1,0) {$1$};
	 \node[vertextype2] (v2) at (1,1.25) {$3$};
	 \node[vertextype2] (v3) at (-1,1.25) {$4$};
	  \node[vertextype2] (v4) at (1,0) {$2$};
	 \node[vertextype2] (v5) at (0,2) {$5$};
	 
	  \node[vertextype2] (v6) at (11.5,.25) {$1$};

	   \node[vertextype2] (v7) at (5.25,3.5) {$1$};
	 \node[vertextype2] (v8) at (7.25,4.75) {$3$};
	 \node[vertextype2] (v77) at (5.25,4.75) {$4$};
	  \node[vertextype2] (v88) at (7.25,3.5) {$2$};

	 
	   \node[vertextype2] (v9) at (5.25,-2.75) {$1$};
	 \node[vertextype2] (v10) at (7.25,-2.75) {$2$};

 	\node[vertextype2] (v11) at (5.25,0) {$1$};
	 \node[vertextype2] (v12) at (7.25,1.25) {$3$};
	  \node[vertextype2] (v14) at (7.25,0) {$2$};
	 
	   \node at (12.25,.25) {${\bf N}_1$};
	    \node at (5.8,4.3) {${\bf N}_4$};
	      \node at (6.25,-2.375) {${\bf N}_2$};
	       \node at (-.45,.775) {${\bf N}_5$};
	         \node at (6,1) {${\bf N}_3$};
	\path[]
	(v1) edge [loop left, blue, thick] node {} (v1)
	(v2) edge [loop right, blue, thick] node {} (v2)
	(v3) edge [loop left, blue, thick] node {} (v3)
	(v4) edge [loop right, blue, thick] node {} (v4)
	(v5) edge [loop right, blue, thick] node {} (v5)
	(v1) edge [thick, blue] node {} (v2)
	(v2) edge [thick, blue] node {} (v3)
	(v1) edge [thick, blue] node {} (v4)
	(v1) edge [thick, blue] node {} (v3)
	(v2) edge [thick, blue] node {} (v5)
	(v4) edge [thick, blue] node {} (v2)
	(v3) edge [blue, thick] node {} (v5)

	(v7) edge [loop left, blue, thick] node {} (v1)
	(v8) edge [loop right, blue, thick] node {} (v2)
	(v77) edge [loop left, blue, thick] node {} (v3)
	(v88) edge [loop right, blue, thick] node {} (v4)
	(v7) edge [thick, blue] node {} (v88)
	(v88) edge [thick, blue] node {} (v8)
	(v7) edge [thick, blue] node {} (v8)
	(v7) edge [thick, blue] node {} (v77)
	(v8) edge [thick, blue] node {} (v77)

	(v6) edge [loop left, blue, thick] node {} (v6)
	
	
	(v9) edge [loop left, blue, thick] node {} (v9)
	(v10) edge [loop right, blue, thick] node {} (v10)
	(v9) edge [thick, blue] node {} (v10)
	
	(v11) edge [loop left, blue, thick] node {} (v11)
	(v12) edge [loop right, blue, thick] node {} (v12)
	(v14) edge [loop right, blue, thick] node {} (v14)
	(v11) edge [thick, blue] node {} (v12)
	(v11) edge [thick, blue] node {} (v14)
	(v14) edge [thick, blue] node {} (v12)
	
	(1.9,.25) edge [thick, decorate, decoration=snake, gray]  node [below, midway] {$a_2$} (4.25,.25) 
	(6.3,-.4) edge [thick, decorate, decoration=snake, gray] node [left, midway] {$a_5$} (6.3,-2)
	(6.3,3) edge [thick, decorate, decoration=snake, gray] node [left, midway] {$a_4$} (6.3,1.4)
	(8.125,.25) edge [thick, decorate, decoration=snake, gray] node [below, midway] {$a_7$} (10.6,.25)
	(1.5,-.5) edge [thick, decorate, decoration=snake, gray] node [below, midway] {$a_3\ $} (4.375,-2.5)
	(8.125,-2.5) edge [thick, decorate, decoration=snake, gray] node [below, midway] {$\ a_8$} (11,-.5)
	(1.5,1.75) edge [thick, decorate, decoration=snake, gray] node [above, midway] {$a_1\ $} (4.2,3.5)
	(8.125,3.5) edge [thick, decorate, decoration=snake, gray] node [above, midway] {$\ \ a_6$} (11,1.25)
	;
 \end{tikzpicture} 
 \vspace{-.2cm}
 \caption{\footnotesize {\rm The subnetwork quiver for the network in Figure \ref{pict1}. Loops from ${\bf N}_i$ to ${\bf N}_i$ are not drawn. Neither are the compositions of arrows that are already depicted. }}
 \vspace{-.6cm}
\label{pict4}
\end{center}
 \end{figure}
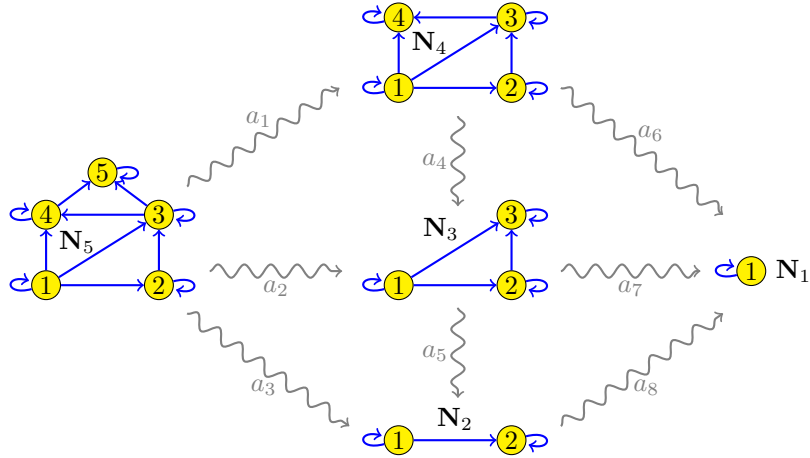 
 The maps $R_{a_i}$ are given by
\begin{align}\nonumber
\begin{array}{ll}
R_{a_1}(x_1, x_2, x_3, x_4, x_5) & =  (x_1, x_2, x_3, x_4),\\
R_{a_2}(x_1, x_2, x_3, x_4, x_5) & =  (x_1, x_2, x_3),\\
R_{a_3}(x_1, x_2, x_3, x_4, x_5) & =  (x_1, x_2),\\
R_{a_4}(x_1, x_2, x_3, x_4) & =  (x_1, x_2, x_3),\\
R_{a_5}(x_1, x_2, x_3) & =  (x_1, x_2),\\
R_{a_6}(x_1, x_2, x_3, x_4) & =  x_1,\\
R_{a_7}(x_1, x_2, x_3) & =  x_1,\\
R_{a_8}(x_1, x_2) & =  x_1.\\
\end{array}
\end{align}
 \end{ex}
\noindent The following result reveals the dynamical meaning of the quiver ${\bf SubQ}({\bf N})$.
\begin{lem}\label{obvious}
Let ${\bf N} = \{E\rightrightarrows^s_t N\}$ be a network and 
$F^{\bf N}:  E_{\bf N} \to E_{\bf N}$ a network map, i.e., it is of the form
$$F^{\bf N}_n\left( \, \bigoplus_{{  m} \in N} x_m\, \right) = F_{  n}\left(  \bigoplus_{{ e} \in { E}\, : \, t(e) = n } x_{ s(e)}   \right) \ \mbox{for all} \ n\in N \, .$$
For any subnetwork ${\bf N}' \sqsubseteq {\bf N}$ define
$F^{{\bf N}'} : E_{{\bf N}'} \to E_{{\bf N}'}$ by 
$$F^{{\bf N}'}_n \left(  \bigoplus_{m\in {\bf N}'} x_m \right) := F_n\left(  \bigoplus_{e \in E \, : \, t(e) = n} x_{s(e)} \right)\ \mbox{for all}\ n\in N'\, .$$
Then  these $F^{{\bf N}'}$ together define a ${\bf SubQ}({\bf N})$-equivariant map.
\end{lem}
\begin{proof} 
First of all, note that the maps $F^{{\bf N}'}$ are well-defined because we assumed that ${\bf N}' \sqsubseteq{\bf N}$, so that $s(e)\in N'$ whenever $t(e)\in N'$. 

To prove ${\bf SubQ}({\bf N})$-equivariance, assume that $a\in A$ is the arrow from ${\bf N'}$ to ${\bf N''}$. It then holds  that ${\bf N}'' \sqsubseteq {\bf N}' \sqsubseteq {\bf N}$, so
$$F^{{\bf N}'}_n\! \left(  \bigoplus_{m\in {\bf N}'} x_m \right) \!\! =  F_n\! \left(  \bigoplus_{e \in E \, : \, t(e) = n} x_{s(e)} \right) \! =  F^{{\bf N}''}_n \! \left(  \bigoplus_{m\in {\bf N}''} x_m \right)\, \mbox{for all}\ n\in N''. $$
But this implies that
\begin{align}\nonumber 
& R_a \left(F^{{\bf N}'} \left(  \bigoplus_{m\in {\bf N}'} x_m \right) \right) = \bigoplus_{n\in N''} F^{{\bf N}'}_n \left(  \bigoplus_{m\in {\bf N}'} x_m \right)  \\ \nonumber 
& = \bigoplus_{n\in N''}  F^{{\bf N}''}_n \left(  \bigoplus_{m\in {\bf N}''} x_m \right) = F^{{\bf N}''} \left( R_a\left( \bigoplus_{m\in {\bf N}'} x_m \right) \right)\, ,
\end{align}
which proves the lemma.
\end{proof}
\noindent   The next result is the converse of Lemma \ref{obvious} and the natural generalisation of Lemma \ref{fflemma}. The proof is similar to that of Lemma \ref{obvious}.
 \begin{lem}\label{alsoobvious}
 A collection of maps $F^{\bf N'}: E_{\bf N'} \to E_{\bf N'}$ (one for each ${\bf N}' \sqsubseteq {\bf N})$ is ${\bf SubQ}({\bf N})$-equivariant if and only if for all ${\bf N}'' \sqsubseteq {\bf N}' \sqsubseteq {\bf N}$ it holds that 
\begin{align}\label{equivariancesubnetworks}
 F^{{\bf N}'}_n \left( \bigoplus_{m\in N'} x_m  \right) = F^{{\bf N}''}_n \left(\bigoplus_{m\in N''} x_m  \right) \ \mbox{for all}\ n\in N'' \, .
 \end{align}
  (In other words: if the $n$-th components of all the maps are equal and depend only on the variables $x_m$ with $m$ in the smallest subnetwork of {\bf N} containing $n$.)
 \end{lem}

 \begin{proof}
Let ${\bf N}'' \sqsubseteq {\bf N}'$ and let $a \in A$ be the arrow with $s(a)={\bf N}'$ and $t(a)={\bf N}''$. By definition of $R_a$, we have on the one hand that
$$ 
R_a \left(F^{{\bf N}'} \left(  \bigoplus_{m\in {\bf N}'} x_m \right) \right) = \bigoplus_{n\in N''} F^{{\bf N}'}_n \left(  \bigoplus_{m\in {\bf N}'} x_m \right)\, .
$$
On the other hand, 
 $$
  F^{{\bf N}''} \left( R_a\left( \bigoplus_{m\in {\bf N}'} x_m \right) \right) = \bigoplus_{n\in N''}  F^{{\bf N}''}_n \left(  \bigoplus_{m\in {\bf N}''} x_m \right) \, .
$$
So $R_a \circ F^{{\bf N}'} = F^{{\bf N}''} \circ R_a$ if and only if (\ref{equivariancesubnetworks}) holds. 
 \end{proof}
\begin{ex}
Let us investigate what  Lemma \ref{alsoobvious} says for the network {\bf N} in Examples \ref{interestingffex} and \ref{ex2quiver}.  So assume that $F^{\bf N_1}, \ldots, F^{\bf N_5}$ form an equivariant map for the quiver depicted in Figure \ref{pict4}. Observe that ${\bf N}_5={\bf N}$ and  that the smallest subnetwork of {\bf N} that contains node $1$ is ${\bf N}_1$. Substituting $n=1$, ${\bf N}'= {\bf N}$ and ${\bf N}''= {\bf N}_1$ in (\ref{equivariancesubnetworks}) yields 
\begin{align}\nonumber
& F^{{\bf N}}_1 (x_1, x_2, x_3, x_4, x_5) = F^{{\bf N}_1}_1 ( x_1) \, .
\end{align}
This shows that $F^{\bf N}_1(x)$ depends only on $x_1$. Continuing in this way for the other nodes, choosing each time for ${\bf N}''$ the smallest subnetworks containing them, we find
\begin{align}\nonumber
& F^{{\bf N}}_2 (x_1, x_2, x_3, x_4, x_5)= F^{{\bf N}_2}_2 ( x_1, x_2) \, , \\ \nonumber
& F^{{\bf N}}_3 (x_1, x_2, x_3, x_4, x_5)=  F^{{\bf N}_3}_3 ( x_1, x_2, x_3) \, ,\\ \nonumber
& F^{{\bf N}}_4 (x_1, x_2, x_3, x_4, x_5)=  F^{{\bf N}_4}_4 ( x_1, x_2, x_3, x_4) \, , \\ \nonumber
& F^{{\bf N}}_5 (x_1, x_2, x_3, x_4, x_5)=  F^{{\bf N}_5}_5 ( x_1, x_2, x_3, x_4, x_5) \, .
\end{align}
We conclude that ${\bf SubQ}({\bf N})$-equivariance is equivalent to $F^{\bf N}$ being of the form 
\begin{align}\label{networkform2}
F^{\bf N} \left( \begin{array}{l} x_1\\ x_2 \\ x_3 \\ x_4 \\ x_5 \end{array} \right) = 
\left( 
\begin{array}{l} 
  F_1^{{\bf N}_1} (x_1) \\
  F_2^{{\bf N}_2} (x_1, x_2) \\
  F_3^{{\bf N}_3} (x_1, x_2, x_3) \\
 F_4^{{\bf N}_4} (x_1, x_2, x_3, x_4) \\
  F_5^{{\bf N}_5} (x_1, x_2, x_3, x_4, x_5) \\
\end{array}
\right)\, 
\end{align}
for some functions $F_i^{{\bf N}_i}$ depending on an appropriate number of variables.  

Note that (\ref{networkform2}) is different from (\ref{networkform}). In fact, any map of the form (\ref{networkform}) is also of the form (\ref{networkform2}) but not vice versa. Hence   (\ref{networkform2}) defines a more general class of maps than (\ref{networkform}). Nevertheless, by construction (\ref{networkform}) and (\ref{networkform2}) have exactly the same subnetworks, so a lot of the network structure of (\ref{networkform}) is also present in (\ref{networkform2}). More importantly, the network structure of (\ref{networkform2}) remains intact when we compose network maps (because quiver-symmetry remains intact under composition, see Proposition \ref{obviouscomp}). We already saw in Example \ref{interestingffex} that  network maps of the form (\ref{networkform}) do not possess this nice property. 
\end{ex}

\section{The quiver of quotient networks}\label{sec5}
Quotient networks were introduced by Golubitsky and Stewart et al. \cite{curious}, \cite{golstew}, \cite{torok},  \cite{stewartnature}, \cite{pivato} to compute robust  synchrony patterns in network dynamical systems. It was shown for the first time in \cite{torok} that every solution of any quotient network lifts to a solution of the original network, i.e., that there is a linear map between the phase spaces that sends solutions of the quotient system to solutions of the original system.
More recently, DeVille and Lerman \cite{deville} generalised this result, and reformulated it using the language of category theory and graph fibrations. The goal of this section is to translate all these observations into the language of quiver representations.

The first half of this section has been added for completeness. We do not aim to provide a comprehensive exposition on quotient networks. Instead, we shall give the basic definitions that allow us to define the quiver of quotient networks. The informed reader may want to skip the first half of this section and start reading from Theorem \ref{devilletheorem1}.
 
We start this section by generalising the notion of a network that was introduced in the previous section. 
\begin{defi}\label{defnetwork}
A {\it coloured network} is a network ${\rm \bf N} = \{E \rightrightarrows^s_t N\}$ in which all nodes and edges are assigned a colour, in such a way that
\begin{itemize}
\item[{\bf 1.}] if two edges $e_1, e_2\in E$ have the same colour, then so do their sources $s(e_1)$ and $s(e_2)$, and so do their targets $t(e_1)$ and $t(e_2)$;
\item[{\bf 2.}] if two nodes $n_1, n_2\in N$ have the same colour, then there is at least one colour preserving bijection  $$\beta_{n_2, n_1}:t^{-1}(n_1)\to t^{-1}(n_2)$$ between the edges that target $n_1$ and $n_2$.
\end{itemize}
\end{defi}
\noindent 
One should think of the networks of Section \ref{sec4} as coloured networks in which all nodes and edges have a different colour, so that conditions {\bf 1} and {\bf 2} are automatically satisfied. 
We remark that the node- and arrow-colours in Definition \ref{defnetwork} are the same as the cell- and arrow-types defined in \cite{torok}. The collection of colour preserving bijections
$$\mathbb{G}_{\bf N}:=\{\, \beta_{n_2, n_1}: t^{-1}(n_1)\to t^{-1}(n_2)\ \mbox{colour preserving bijection} \, | \, n_1, n_2\in N \, \}$$ 
is  the so-called  {\it symmetry groupoid} of Golubitsky, Stewart and Pivato \cite{pivato}. These authors also make the following definition, generalising the network maps that we defined in Section \ref{sec4}.

\begin{defi}\label{defadmissible}
Let ${\rm \bf N} = \{E \rightrightarrows^s_t N\}$ be a coloured network and assume that $F^{\bf N}: \bigoplus_{{  m} \in {  N}} E_{  m}  \to  \bigoplus_{{  m} \in {  N}} E_{  m}$ is a map of the form 
$$F^{\bf N}_n\left( \, \bigoplus_{{  m} \in N} x_m\, \right) = F_{  n}\left(  \bigoplus_{{ e} \in { E}\, : \, t(e) = n } x_{ s(e)}   \right) \, .$$ 
Assume moreover that  
$$E_{n_1}=E_{n_2} \ \mbox{whenever}\ n_1, n_2\in  N \ \mbox{have the same colour},$$ 
and that for every $n_1, n_2\in N$ of the same colour and every colour preserving bijection 
$\beta_{n_2, n_1} \in \mathbb{G}_{\bf N}$ it holds that 
$$F_{n_1}\left(  \bigoplus_{{ e} \in { E}\, : \, t(e) = n_1 } x_{ s(\beta_{n_2, n_1}(e))}   \right) = F_{n_2} \left(  \bigoplus_{{ e} \in { E}\, : \, t(e) = n_2 } x_{ s(e)}   \right)\, .$$
Then we say that $F^{\bf N}$ is an  {\it admissible map} for ${\rm\bf N}$.
\end{defi}
\begin{ex}
Figure  \ref{pict1e} shows an example of a coloured network with two node colours and three edge colours. The edges from a node to itself representing internal dynamics are not depicted. Note that each yellow node is targeted by two blue edges. Hence there are two colour-preserving bijections between the edges targeting any two yellows nodes. Similarly, each green node is targeted by one red and one orange edge, so there is exactly one colour-preserving bijection between the edges targeting any two green nodes.
 
 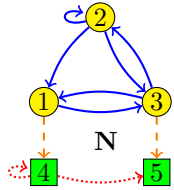
\begin{figure}[h]\renewcommand{\figurename}{\rm \bf \footnotesize Figure} 
\begin{center}
\hspace{0cm}
  \hspace{-.4cm} 
 \begin{tikzpicture}[->, scale=.75]
	 	 \tikzstyle{vertextype2} = [circle, draw, fill=yellow, minimum size=10pt,inner sep=1pt]
		 \tikzstyle{vertextype3} = [rectangle, draw, fill=green, minimum size=10pt,inner sep=1pt]

	 \node[vertextype2] (v1) at (-2,1.5) {$1$};
	 \node[vertextype2] (v2) at (-1,3) {$2$};
	 \node[vertextype2] (v3) at (0,1.5){$3$};
	 \node[vertextype3] (v4) at (-2,0.25) {$4$};
	 \node[vertextype3] (v5) at (0,0.25) {$5$};

	\node at (-0.9, 0.8) {${\bf N}$};
	
	\path[]
	(v1) edge [thick, blue, bend right = 15] node {} (v3)
	(v3) edge [thick, blue, bend right = 15] node {} (v1)
	(v2) edge [thick, blue, bend right = 15] node {} (v3)
	(v3) edge [thick, blue, bend right = 15] node {} (v2)
	(v2) edge [loop left, blue, thick] node {} (v2)
	(v2) edge [thick, blue, bend right = 15] node {} (v1)
	
	(v1) edge [thick, orange, dashed] node {} (v4)
	(v3) edge [thick, orange, dashed] node {} (v5)
	
	(v4) edge [thick, red, bend right = 15, densely dotted] node {} (v5)
	(v4) edge [loop left, red, thick, densely dotted] node {} (v4)
;

 \end{tikzpicture} 
 \vspace{-.1cm}
 \caption{\footnotesize {\rm An example of a network with two node colours and three edge colours. Self-loops describing internal dynamics are not shown.}}
 \vspace{-.5cm}
\label{pict1e}
\end{center}
 \end{figure} 
\noindent An admissible map for this network is of the form
\begin{equation}   \nonumber 
F^{{\bf N}}\left( \begin{array}{c} 
x_1 \\ x_2 \\ x_3 \\ y_4 \\ y_5
\end{array} 
\right) = \left(
 \begin{array}{l}   
    F(x_{1}, \overline{\bl{ x_{2}}, \bl{x_{3}}}) \\
 F(x_{2}, \overline{\bl{x_{2}}, \bl{x_{3}}}) \\
  F(x_{3}, \overline{\bl{x_{1}}, \bl{x_{2}}})\\ 
  G(y_{4}, \ro{y_{4}}, \ora{x_1})\\
 G(y_{5}, \ro{y_{4}}, \ora{x_3})\\
 \end{array}\right) \, ,
 \hspace{-.8cm}
 \end{equation} 
for some functions $F$ and $G$. The bar indicates that variables may be interchanged, i.e., it expresses that $F(x,\overline{\bl{y}, \bl{z}}) = F(x,\overline{\bl{z}, \bl{y}}) $ for all $x,y,z$. 
\end{ex}
\noindent The next  definition is due to DeVille and Lerman  \cite{deville}.
 \begin{defi}\label{deffibration}
Let ${\bf N}= \{E \rightrightarrows^s_t N\}$ and ${\bf N}' = \{E' \rightrightarrows^s_t N'\} $ be coloured networks and let $\phi :{\rm \bf N}\to {\rm \bf N}'$. Assume that 
\begin{itemize}
\item[{\it i)}] this $\phi$ sends edges to edges and nodes to nodes, it preserves the colours of nodes and edges, and sends the head and tail of every edge $e\in E$ to the head and tail of $\phi(e)\in E'$;
 \item[{\it ii)}]  for every node $n \in N$, the restriction $\phi|_{t^{-1}(n)} :t^{-1}(n) \to t^{-1}(\phi(n))$  is a colour preserving bijection.
 \end{itemize}
 Then $\phi$ is called a {\it graph fibration}.
  \end{defi}
  \noindent 
 The key result in  \cite{deville} is the following theorem.
 
 \begin{thr}[DeVille \& Lerman]\label{devilletheorem1}
Let ${\rm \bf N} = \{E \rightrightarrows^s_t N\}$ and ${\rm \bf N}' = \{E' \rightrightarrows^s_t N'\}$ be coloured networks, let $\phi: {\rm \bf N} \to {\rm \bf N}'$ be a graph fibration, and let  $F^{\bf N}$ and $F^{{\bf N}'}$ be admissible maps for ${\bf N}$ and ${\bf N}'$ respectively. In particular, they have the form
\begin{align} \nonumber 
F^{\bf N}_n\left( \, \bigoplus_{{  m} \in N} x_m\, \right) & = F_{ n}\left(  \bigoplus_{{ e} \in { E}\, : \, t(e) = n } x_{ s(e)}   \right) \ \mbox{and} \\ \nonumber   
F^{{\bf N}'}_{n'}\left( \, \bigoplus_{{  m} \in N'} x_m\, \right) & = F'_{ n'}\left(  \bigoplus_{{ e'} \in { E'}\, : \, t(e') = n' } x_{ s(e')}   \right)\, .
\end{align}
Finally, assume that for every $n\in N, n'\in N'$ of the same colour and every colour preserving bijection $\beta_{n', n}: t^{-1}(n) \to t^{-1}(n')$, it holds that
$$F_{n} \left(  \bigoplus_{{ e} \in { E}\, : \, t(e) = n } x_{ s(\beta_{n', n}(e))}   \right)  = F'_{n'}\left(  \bigoplus_{{ e'} \in { E'}\, : \, t(e') = n' } x_{ s(e')}   \right)\, .$$
Then the linear map 
$$R_{\phi}: \bigoplus_{m'\in {\bf N}'} E_{m'} \to \bigoplus_{m\in {\bf N}} E_{m} \ \mbox{defined by}\ R_{\phi} \left( \bigoplus_{m'\in N'} x_{m'} \right) := \bigoplus_{m\in N} x_{\phi(m)}\, .$$
satisfies
$$R_{\phi} \circ F^{{\bf N}'} = F^{\bf N}  \circ R_{\phi} \, .$$
 \end{thr}
\noindent  The proof of the theorem is simple and consists of combining all the definitions that were made. It can be found in \cite{deville}. 

It is not hard to see that ${\bf N'}\sqsubseteq {\bf N}$ is a subnetwork if and only if the inclusion $i: {\bf N'}\to {\bf N}$ is an injective graph fibration. The map 
$R_{i}: \bigoplus_{m\in N} E_m \to \bigoplus_{m\in N'} E_m$ is then given by 
$$R_{i}\left(\bigoplus_{m\in N} x_m\right) = \bigoplus_{m\in N'} x_{i(m)} =\bigoplus_{m\in N'} x_{m}\, .$$ 
So we recover the linear maps of Section \ref{sec4}. In this section we shall be interested in surjective graph fibrations instead.
\begin{defi}
When $\phi:{\rm \bf N}\to {\rm \bf N}'$ is a surjective graph fibration, then we call ${\rm \bf N}'$ a {\it quotient} of ${\rm \bf N}$.
\end{defi}
\noindent We are now ready to define the quiver of quotient networks.
 \begin{defi}
Let ${\bf N}$ be a coloured network. The quiver ${\bf QuoQ}({\bf N}) = \{A \rightrightarrows^s_t V\}$ of quotient networks of ${\bf N}$ has as its  vertices the nonempty quotients of ${\bf N}$, i.e.,
$$V=\{ {\bf N}' \neq \emptyset \, | {\bf N}'  \ \mbox{is a quotient of} \ {\bf N}  \} \,  .$$ 
There is exactly one arrow $a \in A$  with $s(a)={\bf N}'$ and $t(a)={\bf N}''$ for each distinct surjective graph fibration $\phi$ from ${\bf N}''$ to ${\bf N}'$. 
\end{defi}
\noindent A representation of ${\bf QuoQ}({\bf N})$ is defined in a straightforward manner. To each quotient ${\bf N}'$ of ${\bf N}$ (i.e., each vertex ${\bf N}'$ of the quiver ${\bf QuoQ}({\bf N})$), we assign the vector space
 $$E_{{\bf N}'} := \bigoplus_{m \in N'} E_{m}\, , $$
and for each arrow $a \in A$ from ${\bf N}'$ to ${ \bf N}''$ (corresponding to the graph fibration $\phi: {\bf N}'' \to {\bf N}'$), we define $R_a = R_{\phi}$, where $R_{\phi}$ is the linear map defined in Theorem \ref{devilletheorem1}. In other words, $R_a: E_{{\bf N}'}   \to E_{{\bf  N}''}$ is defined by the formula 
 \begin{align}
 R_{a} \left( \bigoplus_{m\in N'} x_m \right) := \bigoplus_{m\in N''} x_{\phi(m)}\, . \label{repmap}
 \end{align}
Theorem \ref{devilletheorem1} then trivially translates into the following result.
  \begin{cor} 
  Let ${\rm \bf N} = \{E \rightrightarrows^s_t N\}$ be a coloured network and let $F^{\bf N} : E_{{\bf N}}  \to E_{{\bf N}} $ be an admissible map, so that in particular it is of the form
  $$F^{\bf N}_n\left( \, \bigoplus_{{  m} \in N} x_m\, \right) = F_{ n}\left(  \bigoplus_{{ e} \in { E}\, : \, t(e) = n } x_{ s(e)}   \right)\, .$$
  For each surjective graph fibration $\phi: {\bf N} \to {\bf N}'$, define $F^{\bf N'}: E_{{\bf N}'}  \to E_{{\bf N}'} $ by 
  $$F^{{\bf N}'}_{\phi(n)}\left( \, \bigoplus_{{  m} \in N'} x_m\, \right) := F_{ n}\left(  \bigoplus_{{ e} \in { E}\, : \, t(e) = n } x_{ s(\phi(e))}   \right)\, .$$
  Then each $F^{{\bf N}'}$ is well-defined and admissible for ${\bf N}'$. Together the $F^{{\bf N}'}$ $({\bf N}'$ quotient of ${\bf N})$ form a ${\bf QuoQ}({\bf N})$-equivariant map.
  \end{cor}
\begin{ex}
The network {\bf N} in Figure \ref{pict1e} has six nonempty quotients (including ${\bf N}={\bf N}_1$ itself).  Figure \ref{pict2e} shows the  quiver of quotient networks. 
 
 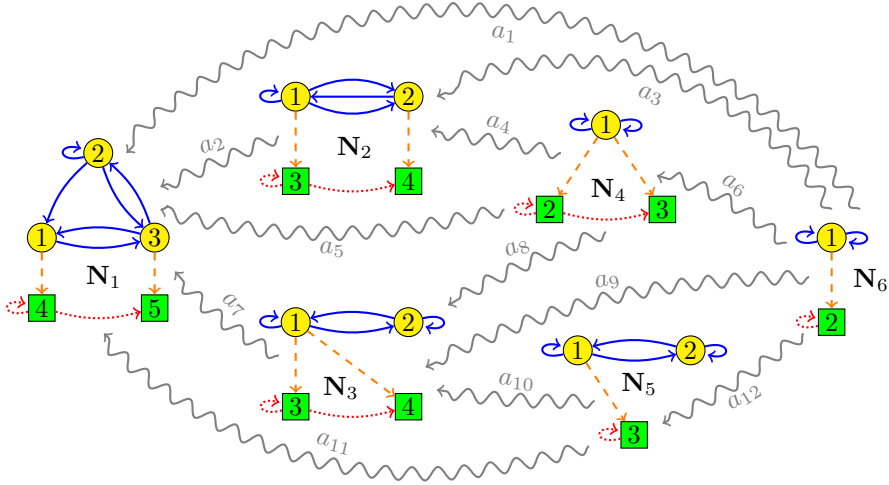
\begin{figure}[h]\renewcommand{\figurename}{\rm \bf \footnotesize Figure} 
\begin{center}
\hspace{0cm}
  \hspace{-.4cm} 
 \begin{tikzpicture}[->, scale=.75]
	 	 \tikzstyle{vertextype2} = [circle, draw, fill=yellow, minimum size=10pt,inner sep=1pt]
		 \tikzstyle{vertextype3} = [rectangle, draw, fill=green, minimum size=10pt,inner sep=1pt]

	 \node[vertextype2] (v1) at (-2,1.5) {$1$};
	 \node[vertextype2] (v2) at (-1,3) {$2$};
	 \node[vertextype2] (v3) at (0,1.5){$3$};
	 \node[vertextype3] (v4) at (-2,0.25) {$4$};
	 \node[vertextype3] (v5) at (0,0.25) {$5$};

	 \node[vertextype2] (v11) at (2.5,0) {$1$};
	 \node[vertextype2] (v12) at (4.5,0) {$2$};
	 \node[vertextype3] (v13) at (2.5,-1.5) {$3$};
	 \node[vertextype3] (v14) at (4.5,-1.5) {$4$};
	 
	 \node[vertextype2] (v21) at (2.5,4) {$1$};
	 \node[vertextype2] (v22) at (4.5,4) {$2$};
	 \node[vertextype3] (v23) at (2.5,2.5) {$3$};
	 \node[vertextype3] (v24) at (4.5,2.5) {$4$};  

	 \node[vertextype2] (v31) at (8,3.5) {$1$};
	 \node[vertextype3] (v32) at (7, 2) {$2$};
	 \node[vertextype3] (v33) at (9, 2) {$3$};

	 \node[vertextype2] (v41) at (7.5,-0.5) {$1$};
	 \node[vertextype2] (v42) at (9.5,-0.5) {$2$};
	 \node[vertextype3] (v43) at (8.5,-2) {$3$};

	 \node[vertextype2] (v51) at (12, 1.5) {$1$};
	 \node[vertextype3] (v52) at (12, 0) {$2$};

	\node at (-0.9, 0.8) {${\bf N}_1$};
	\node at (3.55, 3.1) {${\bf N}_2$};
	\node at (3.3, -1.16) {${\bf N}_3$};
	\node at (8.59,-1.1) {${\bf N}_5$};
	\node at (8.05, 2.35) {${\bf N}_4$};
	\node at (12.7, 0.75) {${\bf N}_6$};
	
	\path[]
	(v1) edge [thick, blue, bend right = 15] node {} (v3)
	(v3) edge [thick, blue, bend right = 15] node {} (v1)
	(v2) edge [thick, blue, bend right = 15] node {} (v3)
	(v3) edge [thick, blue, bend right = 15] node {} (v2)
	(v2) edge [loop left, blue, thick] node {} (v2)
	(v2) edge [thick, blue, bend right = 15] node {} (v1)
	
	(v11) edge [loop left, blue, thick] node {} (v11)
	(v12) edge [loop right, blue, thick] node {} (v12)
	(v11) edge [thick, blue, bend right = 15] node {} (v12)
	(v12) edge [thick, blue, bend right = 15] node {} (v11)
	
	(v21) edge [loop left, blue, thick] node {} (v21)
	(v21) edge [thick, blue, bend left = 25] node {} (v22)
	(v22) edge [thick, blue, bend right = 0] node {} (v21)
	(v21) edge [thick, blue, bend right = 25] node {} (v22)
	
	(v31) edge [loop right, blue, thick] node {} (v31)
	(v31) edge [loop left, blue, thick] node {} (v31)
	
	(v41) edge [loop left, blue, thick] node {} (v41)
	(v42) edge [loop right, blue, thick] node {} (v42)
	(v41) edge [thick, blue, bend right = 15] node {} (v42)
	(v42) edge [thick, blue, bend right = 15] node {} (v41)
	
	(v51) edge [loop right, blue, thick] node {} (v51)
	(v51) edge [loop left, blue, thick] node {} (v51)	
	
	(v1) edge [thick, orange, dashed] node {} (v4)
	(v3) edge [thick, orange, dashed] node {} (v5)

	(v11) edge [thick, orange, dashed] node {} (v13)
	(v11) edge [thick, orange, dashed] node {} (v14)
	
	(v21) edge [thick, orange, dashed] node {} (v23)
	(v22) edge [thick, orange, dashed] node {} (v24)
	
	(v31) edge [thick, orange, dashed] node {} (v32)
	(v31) edge [thick, orange, dashed] node {} (v33)
	
	(v41) edge [thick, orange, dashed] node {} (v43)
	
	(v51) edge [thick, orange, dashed] node {} (v52)
	
	(v4) edge [thick, red, bend right = 15, densely dotted] node {} (v5)
	(v4) edge [loop left, red, thick, densely dotted] node {} (v4)
	
	(v13) edge [thick, red, bend right = 15, densely dotted] node {} (v14)
	(v13) edge [loop left, red, thick, densely dotted] node {} (v13)
	
	(v23) edge [thick, red, bend right = 15, densely dotted] node {} (v24)
	(v23) edge [loop left, red, thick, densely dotted] node {} (v23)
	
	(v32) edge [thick, red, bend right = 15, densely dotted] node {} (v33)
	(v32) edge [loop left, red, thick, densely dotted] node {} (v32)
	
	(v43) edge [loop left, red, thick, densely dotted] node {} (v43)
	
	(v52) edge [loop left, red, thick, densely dotted] node {} (v52)

	(12.5, 2) edge [thick, decorate, decoration=snake, bend left = -50, gray]  node [below, midway, sloped] {$a_{1}$} (-0.5, 3) 
	(2.2, 3.3) edge [thick, decorate, decoration=snake, gray]  node [above, midway, sloped] {$a_2$} (0.1, 2.4) 
	(12, 2) edge [thick, decorate, decoration=snake, bend left = -40, gray]  node [below, midway, sloped] {$a_{3}$} (5, 4) 
	(7.2, 3) edge [thick, decorate, decoration=snake, gray]  node [above, midway, sloped] {$a_4$} (4.9, 3.4) 
	(6.2, 2) edge [thick, decorate, decoration=snake, bend left = 13, gray]  node [below, midway, sloped] {$a_5$} (0.1, 2) 
	(11.2, 1.4) edge [thick, decorate, decoration=snake, bend left = -12, gray]  node [above, midway, sloped] {$a_6$} (8.9, 2.6) 
	(2.2, -0.6) edge [thick, decorate, decoration=snake, bend left = 10, gray]  node [above, midway, sloped] {$a_7$} (0.36, 0.9) 
	(8, 1.6) edge [thick, decorate, decoration=snake, bend left = -5, gray]  node [above, midway, sloped] {$a_8$} (5.2, 0.3) 
	(11.6, 0.8) edge [thick, decorate, decoration=snake, bend left = -12, gray]  node [above, midway, sloped] {$a_9$} (4.8, -0.8) 
	(7.8, -1.4) edge [thick, decorate, decoration=snake, bend left = 5, gray]  node [above, midway, sloped] {$a_{10}$} (5, -1.2) 
	(7.7, -2.2) edge [thick, decorate, decoration=snake, bend left = 30, gray]  node [above, midway, sloped] {$a_{11}$} (-0.9, -0.2) 
	(11.5, -0.2) edge [thick, decorate, decoration=snake, bend left = 5, gray]  node [below, midway, sloped] {$a_{12}$} (9, -1.8) 

	;
 \end{tikzpicture} 
 \vspace{-.3cm}
 \caption{\footnotesize {\rm The quiver of quotient networks for the network in Figure \ref{pict1e}. Loops from ${\bf N}_i$ to ${\bf N}_i$ are not drawn.}}
 \vspace{-.6cm}
\label{pict2e}
\end{center}
 \end{figure} 
 To illustrate, note that there is a graph fibration $\phi_2: {\bf N}_1 \to {\bf N}_2$ which sends node $1$ and $2$ to node $1$, node $3$ to node $2$, node $4$ to node $3$ and node $5$ to node $4$. The corresponding linear map in the representation given by formula (\ref{repmap}) is
$$R_{a_2}(x_1, x_2, y_3, y_4) = (x_1,x_1,x_2, y_3, y_4)\, .$$
The admissible maps for ${\bf N}_1$ and ${\bf N}_2$ are given by
\begin{equation}   \nonumber 
 F^{{\bf N}_1}\left( \begin{array}{c} 
x_1 \\ x_2 \\ x_3 \\ y_4 \\ y_5
\end{array} 
\right) = \left(
 \begin{array}{l}   
    F(x_{1}, \overline{\bl{ x_{2}}, \bl{x_{3}}}) \\
 F(x_{2}, \overline{\bl{x_{2}}, \bl{x_{3}}}) \\
  F(x_{3}, \overline{\bl{x_{1}}, \bl{x_{2}}})\\ 
  G(y_{4}, \ro{y_{4}}, \ora{x_1})\\
 G(y_{5}, \ro{y_{4}}, \ora{x_3})\\
 \end{array}\right) \, \mbox{and}\ 
F^{{\bf N}_2}\left( \begin{array}{c} 
x_1 \\ x_2  \\ y_3 \\ y_4
\end{array} 
\right) = \left(
 \begin{array}{l}   
    F(x_{1}, \overline{\bl{ x_{1}}, \bl{x_{2}}}) \\
 F(x_{2}, \overline{\bl{x_{1}}, \bl{x_{1}}}) \\
  G(y_{3}, \ro{y_{3}}, \ora{x_1})\\
 G(y_{4}, \ro{y_{3}}, \ora{x_2})\\
 \end{array}\right) \, .
 \hspace{-.8cm}
 \end{equation} 
 One verifies that indeed $R_{a_2} \circ F^{{\bf N}_2} = F^{{\bf N}_1} \circ R_{a_2}$. 
 
 The full list of representation maps for the arrows in Figure \ref{pict2e} is given by
\begin{align}\nonumber
\begin{array}{ll}
R_{a_1}(x_1, y_2) &= (x_1,x_1,x_1, y_2, y_2)\, , \\  
R_{a_2}(x_1, x_2, y_3, y_4) &= (x_1,x_1,x_2, y_3, y_4)\, , \\ 
R_{a_3}(x_1, y_2) &= (x_1,x_1, y_2, y_2)\, , \\ 
R_{a_4}(x_1, y_2, y_3) &= (x_1,x_1, y_2, y_3)\, , \\  
R_{a_5}(x_1, y_2, y_3) &= (x_1,x_1, x_1, y_2, y_3)\, , \\ 
R_{a_6}(x_1, y_2) &= (x_1, y_2, y_2)\, , \\  
R_{a_7}(x_1, x_2, y_3, y_4) &= (x_1,x_2, x_1, y_3, y_4)\, ,  \\ 
R_{a_8}(x_1, y_2, y_3) &= (x_1, x_1, y_2, y_3)\, , \\  
R_{a_9}(x_1, y_2) &= (x_1, x_1, y_2, y_2) \, ,  \\ 
R_{a_{10}}(x_1, x_2, y_3) &= (x_1, x_2, y_3, y_3)\, , \\  
R_{a_{11}}(x_1, x_2, y_3) &= (x_1, x_2, x_1, y_3, y_3)\, , \\ 
R_{a_{12}}(x_1, y_2) &= (x_1, x_1, y_2)\, . 
\end{array}
\end{align}
\end{ex}

\section{Endomorphisms of quiver representations}\label{sec6}
In this section, we gather some basic properties of endomorphisms of quiver representations that will be important in the remainder of this paper.
An endomorphism is simply a linear equivariant map:
\begin{defi}
An {\it endomorphism} of a quiver representation ({\bf E}, {\bf R}) of a quiver ${\bf Q}= \{A\rightrightarrows^s_t V\}$ is a set {\bf L} of {\it linear} maps $L_v: E_v\to E_v$ (one for each  $v\in V$) such that 
$$ L_{t(a)} \circ R_a = R_a \circ L_{s(a)} \ \ \mbox{for every arrow}\ a\in A\, .$$
The collection of all endomorphisms is denoted by ${\rm End}({\bf E}, {\bf R})$.
\end{defi}
\begin{ex}
For any representation ({\bf E}, {\bf R}) of any quiver ${\bf Q}= \{A\rightrightarrows^s_t V\}$, the identity {\bf Id}, consisting of the maps ${\rm Id}_v: E_v \to E_v$ ($v\in V$), is an example of an endomorphism. This is simply because $ {\rm Id}_{t(a)} \circ R_a = R_a \circ {\rm Id}_{s(a)}$.
\end{ex}
\begin{ex}
If ${\bf F} \in C^{\infty}({\bf E}, {\bf R})$ is a smooth equivariant map of a representation $({\bf E}, {\bf R})$ and ${\bf F}(0)=0$ (meaning that $F_v(0)=0$ for every $v\in V$) then the derivative ${\bf L} = D{\bf F}(0)$ (consisting of the maps $L_v:=DF_v(0): E_v\to E_v$) is an example of an endomorphism. This follows from differentiating the identities $ F_{t(a)} \circ R_a = R_a \circ F_{s(a)}$ at $0$ and noting that $R_a(0)=0$.
\end{ex}
\begin{defi}
A {\it subrepresentation} ${\bf D}$ of a representation ({\bf E}, {\bf R}) of a quiver ${\bf Q}= \{A\rightrightarrows^s_t V\}$ is a set of linear subspaces $D_{v} \subset E_{v}$ ($v\in V$) such that 
$$R_a(D_{s(a)}) \subset D_{t(a)}  \ \ \mbox{for every arrow}\ a\in A\, .$$
\end{defi}
\noindent In other words, {\bf D} is a subrepresentation of ({\bf E}, {\bf R}) if the restriction $({\bf D}, {\bf R}|_{\bf D})$ defines a representation. Examples of subrepresentations are the eigenspaces of endomorphisms. In this paper, we will use generalised eigenspaces more often than eigenspaces, so we formulate the following as a separate proposition.

\begin{prop}\label{eigenspacesprop}
Let ${\bf L}$ be an endomorphism of a representation $({\bf E}, {\bf R})$ of a quiver ${\bf Q}= \{A\rightrightarrows^s_t V\}$. We say that $\lambda\in \mathbb{C}$ is an eigenvalue of ${\bf L}$ if there is at least one  $v\in V$ such that $\lambda$ is an eigenvalue of $L_v$. We then call $\lambda$ an eigenvalue of all the $L_{w}$ ($w\in V$) even if the corresponding eigenspace of $L_w$ is trivial.
\begin{itemize}
\item[{\it i)}] For $\lambda \in \mathbb{R}$, denote by $E^{\lambda}_v\subset E_v$ the generalised eigenspace of $L_v:E_v \to E_v$ for the eigenvalue $\lambda$. Then the $E_v^{\lambda}$ define a subrepresentation ${\bf E}^{\lambda}$ of $({\bf E}, {\bf R})$. 
\item[{\it ii)}] For $\mu \in \mathbb{C}\backslash \mathbb{R}$, denote by $E^{\mu, \bar \mu}_v \subset E_v$ the real generalised eigenspace of $L_v: E_v \to E_v$ for the eigenvalue pair $\mu, \bar \mu$. Then the $E_v^{\mu, \bar \mu}$ define a subrepresentation ${\bf E}^{\mu, \bar \mu}$ of $({\bf E}, {\bf R})$. 
\end{itemize}
\end{prop}
\begin{proof}
Recall that {\bf L} consists of linear maps $L_v: E_v\to E_v$ $(v\in V)$ for which $R_a \circ L_{s(a)} = L_{t(a)} \circ R_a$ for each $a\in A$. Choose $\lambda \in \mathbb{R}$ and assume that $x\in E_{s(a)}^{\lambda}$. This means that $(L_{s(a)} - \lambda {\rm Id}_{s(a)})^N (x) = 0$ for any $N\geq {\rm dim}\, E_{s(a)}$. But then 
$$(L_{t(a)} - \lambda {\rm Id}_{t(a)})^N (R_a x) = R_a (L_{s(a)} - \lambda {\rm Id}_{s(a)})^N (x)=0\, .$$ 
So $R_a(E_{s(a)}^{\lambda}) \subset E_{t(a)}^{\lambda}$. 
For $\mu \in \mathbb{C}\backslash \mathbb{R}$, $E_v^{\mu, \bar \mu} = \ker (({ L}_v - \mu {\rm Id}_v)({ L}_v - \overline{\mu} {\rm Id}_v))^N$. So the proof is completely analogous.
\end{proof}

\section{Lyapunov-Schmidt reduction and quivers}\label{sec7}
In this and the coming sections, we will show that quiver symmetry can be preserved in a number of well-known dimension reduction techniques. We start with the most straightforward result, which shows that quiver symmetry can be preserved in the process of Lyapunov-Schmidt reduction. We only prove this  for steady state bifurcations at this point. How to preserve quiver symmetry in the Lyapunov-Schmidt reduction for periodic orbits is left as an open problem.
 
 Let us start by reviewing the classical Lyapunov-Schmidt reduction process for steady state bifurcations (so without any quiver symmetry) to set the stage for the proof of Theorem \ref{LStheorem} below. We consider the differential equation 
 $$\frac{dx}{dt} = F(x;\lambda) \ \mbox{for}\ x\in E \ \mbox{and}\ \lambda \in \Lambda \subset \mathbb{R}^p\, , $$
where   $F:E \times \Lambda \to E$ is a smooth vector field defined on a finite dimensional vector space $E$, depending smoothly on parameters from an open set $\Lambda \subset \mathbb{R}^p$. We also assume that for some value of the parameters this differential equation admits a steady state. We assume without loss of generality that $F(0; 0)=0$. The goal is to find all other steady states near $(x; \lambda)=(0; 0)$ by reducing the equation
$$F(x; \lambda)=0 \ \mbox{on}\ E\times \Lambda$$ 
to a simpler ``bifurcation equation'' with as few dimensions as possible.

To explain how this is done, denote by $L = D_xF(0;0): E\to E$ the derivative of $F$ in the direction of $E$ at $(0;0)$. We shall denote by $E^{\rm ker}\subset E$ the generalised kernel of $L$ (i.e., the generalised eigenspace for the eigenvalue zero) and by $E^{\rm im}$ its reduced image (the sum of the remaining generalised eigenspaces). We write 
$$\pi: E = E^{\rm im} \oplus E^{\rm ker} \to E^{\rm im}\, , \  x = x^{\rm im} + x^{\rm ker} \mapsto x^{\rm im}$$ 
for  the projection onto $E^{\rm im}$ along $E^{\rm ker}$ (i.e., $\pi$ has kernel $E^{\rm ker}$ and is the identity on $E^{\rm im}$). 
The derivative in the direction of $E^{\rm im}$ of
$$\pi  \circ F : E^{\rm im} \oplus E^{\rm ker} \times \Lambda \to E^{\rm im}$$
at $(0;0)$ is equal to $$D_{x^{\rm im}}(\pi \circ F)(0;0)   = \pi \circ L |_{E^{\rm im}} : E^{\rm im} \to E^{\rm im}\, .$$
By construction this map is invertible. By the implicit function theorem there is thus a unique smooth function
$$\phi: U \subset E^{\rm ker} \times \Lambda   \to W \subset E^{\rm im}$$ 
defined on some open neighbourhood $U$ of $(0;0)\in E^{\rm ker} \times \Lambda$ and mapping into an open neighborhood $W$ of $0\in E^{\rm im}$ that satisfies $$(\pi \circ F) (x^{\ker} + \phi(x^{\ker}; \lambda); \lambda)=0\, .$$ 
We clearly have $\phi(0;0)=0$ because $F(0;0)=0$. To find all other solutions $(x^{\rm im}, x^{\rm ker}; \lambda)\in W \times U$ to the equation $F(x^{\rm im}, x^{\rm ker}; \lambda)=0$, it then remains to solve only the reduced bifurcation equation
\begin{align}\label{redLS}
f(x^{\ker}; \lambda):= ((1 - \pi ) \circ F )(x^{\ker} +\phi(x^{\ker}; \lambda); \lambda) = 0\, ,
\end{align}
where
$$f: U \subset E^{\rm ker}  \times \Lambda   \to E^{\rm ker}\, .$$ 
This method to (locally) reduce the equation $F(x; \lambda)=0$ to the lower-dimensional equation $f(x^{\ker}; \lambda)=0$ is called Lyapunov-Schmidt reduction. 
 The following theorem states that the reduced equation inherits  quiver symmetry if it is present in the original equation.
 

\begin{thr} \label{LStheorem}(Quiver equivariant Lyapunov-Schmidt theorem) Let $({\bf E}, {\bf R})$ be a representation of a quiver ${\bf Q}=\{A\rightrightarrows^s_t V\}$ and assume that ${\bf F}\in C^{\infty}({\bf E} \times \Lambda, {\bf R})$ is a smooth parameter-dependent {\bf Q}-equivariant map, i.e., for every $v\in V$ there is a smooth $F_v:E_v\times \Lambda \to E_v$ satisfying 
$$R_a ( F_{s(a)}(x; \lambda)) = F_{t(a)}(R_a(x); \lambda)   \ \mbox{for all} \ x\in E_{s(a)}, \lambda\in \Lambda\ \mbox{and}\ a\in A\, .$$ 
Assume moreover that ${\bf F}(0;0)=0$, i.e., $F_v(0;0)=0$ for all $v\in V$. 

Then the reduced maps $f_v: U_v \subset E_v^{\rm ker} \times \Lambda \to E_v^{\rm ker}$ ($v\in V$) defined in (\ref{redLS}) satisfy 
$$R_a ( f_{s(a)}(x^{\ker}; \lambda)) = f_{t(a)}(R_a(x^{\ker}); \lambda)   \ \mbox{for all} \ a\in A\, $$ 
and for all $(x^{\ker};\lambda) \in \overline U_{s(a)}$ in some open neighbourhood $\overline U_{s(a)}$ of $(0;0)$. 

This means that the $f_v\ (v\in V)$ define a {\bf Q}-equivariant map ${\bf f}$ on an open neighbourhood of $(0; 0)$ of the subrepresentation ${\bf E}^{\ker}\times \Lambda$ of $({\bf E}\times \Lambda, {\bf R})$.
\end{thr} 
\begin{proof}
Fix an $a\in A$ and consider the map $R_a: E_{s(a)}\to E_{t(a)}$. Recall that $R_a \circ L_{s(a)} = L_{t(a)} \circ R_a$, where $L_v = D_xF_v(0;0)$, so that $R_a(E_{s(a)}^{\ker}) \subset E_{t(a)}^{\ker}$ and  $R_a(E_{s(a)}^{\rm im}) \subset E_{t(a)}^{\rm im}$ by Proposition \ref{eigenspacesprop}. It follows in particular that 
$$R_a \circ \pi_{s(a)} = \pi_{t(a)} \circ R_a\, .$$ 
Recall that by definition of $\phi_{s(a)}: U_{s(a)} \to W_{s(a)}$ it holds that 
$$(\pi_{s(a)} \circ F_{s(a)})( x^{\ker} + \phi_{s(a)}(x^{\ker}; \lambda); \lambda) = 0\, $$
for all $(x^{\ker}; \lambda)\in U_{s(a)}$. It follows that
\begin{align}\nonumber
0 & = (R_a \circ \pi_{s(a)} \circ F_{s(a)}) ( x^{\ker} + \phi_{s(a)}(x^{\ker}; \lambda); \lambda) \\ \nonumber & = (\pi_{t(a)} \circ F_{t(a)}) ( R_a (x^{\ker}) + R_a (\phi_{s(a)}(x^{\ker}; \lambda )); \lambda)\, .
\end{align}
By definition of $\phi_{t(a)} : U_{t(a)} \to W_{t(a)}$ it thus holds that 
$$\phi_{t(a)}(R_a (x^{\ker}); \lambda) = R_a (\phi_{s(a)}(x^{\ker}; \lambda) )\, $$
for all $(x^{\ker}; \lambda)\in U_{s(a)}$ with $(R_a(x^{\ker}); \lambda)\in U_{t(a)}$ and $R_a(\phi_{s(a)}(x^{\ker}; \lambda))\in W_{t(a)}$. The $(x^{\ker};\lambda)$ for which these inclusions hold form an  open neighbourhood $ \tilde U_{a}$ of $(0;0)$. 
For  $(x^{\ker}; \lambda) \in  \tilde U_{a}$ we  then have that 
\begin{align}\nonumber
&R_a (f_{s(a)} (x^{\ker}; \lambda))=  (R_a \circ ( 1- \pi_{s(a)}) \circ F_{s(a)} )(x^{\ker}+\phi_{s(a)}(x^{\ker}; \lambda); \lambda ) \\ \nonumber
&= (( 1 - \pi_{t(a)}) \circ F_{t(a)} )(R_a (x^{\ker})+\phi_{t(a)}(R_a(x^{\ker}); \lambda);\lambda) = f_{t(a)} (R_a (x^{\ker}); \lambda)\, . 
\end{align}
This would prove the theorem if for every vertex $v\in V$ there was at most one arrow $a\in A$ with $s(a)=v$. If there are more such arrows, then the finite intersection $\overline{U}_v := \bigcap_{a: s(a)=v} \tilde U_{a}$  will satisfy  the requirements.
\end{proof}

\section{Center manifolds and quivers} \label{sec8}
In this section we show that quiver symmetry can be preserved in the process of center manifold reduction. The main result is Theorem \ref{CMtheorem} below, which  is a {\bf Q}-equivariant {\it global} center manifold theorem. We encountered various obstructions in trying to prove a fully general ${\bf Q}$-equivariant {\it local} center manifold theorem. These will be discussed in Remark \ref{localproblem} below.

We start our analysis by  recalling the classical global center manifold theorem \cite{vdbauw}. We will not prove this classical theorem here, and for simplicity we only formulate a version of the theorem without parameters.
To formulate the classical result, let $E$ be a finite dimensional real vector space and $L: E\to E$ a linear map. Let us denote by $E^{ c}$ the center subspace of $L$ (the sum of the generalised eigenspaces of $L$ for the eigenvalues on the imaginary axis) and by $E^{ h}$ the hyperbolic subspace of $L$ (the sum of the generalised eigenspaces of $L$ for the eigenvalues not lying on the imaginary axis). We shall denote by 
$$\pi^c : E = E^{  c} \oplus E^{h} \to E^{ c}\ \mbox{and by}\ \pi^h: = 1-\pi^c : E = E^{ c} \oplus E^{ h} \to E^{ h}$$ 
the projections corresponding to the splitting $E = E^{ c} \oplus E^{ h}$. Now we can formulate the global center manifold theorem, referring to \cite{vdbauw} for a proof.

\begin{thr}\label{cmthm}
 Let $L:E\to E$ be a linear map and $k \in \{1,2,3, \ldots\}$. Then there is an $\varepsilon = \varepsilon(L, k) >0$ for which the following holds. 

If $F: E\to E$ is a $C^k$ vector field that satisfies $F(0)=0$, $DF(0)=L$,
$$  \sup_{x\in E} ||D^{\alpha} (F(x)-L) || < \infty \ \mbox{for all}\ |\alpha| \leq k \ \mbox{and}\ \sup_{x\in E} ||DF(x)-L|| < \varepsilon\, ,$$
 then there exists a $C^k$ map $\phi: E^{ c} \to  E^{ h}$, satisfying $\phi(0)=0$ and $D\phi(0)=0$, of which the graph 
 $$M^{ c}:= \{ x^c + \phi(x^c) \ |\ x^c \in E^c\} \subset E$$
 is an invariant manifold for the flow of the differential equation
$\frac{dx}{dt} = F(x)$. Moreover, if we denote this flow by $e^{tF}$, then
$$M^{ c} = \left\{ x\in E\ \left| \ \sup_{t\in \mathbb{R}  } || (\pi^h \circ e^{tF})(x) || < \infty\, \right. \right\} .$$
We call $M^{ c}$ the {\rm global center manifold} of $F$.  
 \end{thr}

\begin{remk}
Let $x(t)$ be an integral curve of $F$, i.e., $\frac{dx(t)}{dt} = F(x(t))$, and let us write $x^c(t):=\pi^c( x(t))$. Then 
$$\frac{dx^c(t)}{dt} = (\pi^c\circ F)(x(t)) \, .$$ 
If $x(t)$ happens to lie inside $M^{ c}$, then by definition of $\phi$ we moreover have that
$x(t) =  x^c(t) + \phi (x^c(t))$.
 So then
$$\frac{dx^c(t)}{dt} =  (\pi^c\circ F)  (x^c(t) + \phi(x^c(t)))\, .$$
This proves that the restriction of $\pi^c$ to $M^{ c}$ sends integral curves of $F$ in $M^c$ to integral curves of the vector field $F^c: E^{  c} \to E^{  c}$ defined by 
$$F^{  c}(x^c) := (\pi^c \circ F)(x^c+\phi(x^c))\, .$$
We shall call this vector field $F^{ c}$ on $E^{ c}$ the {\it center manifold reduction} of $F$. 
\end{remk}

\noindent We are now ready to formulate our result on quivers and center manifolds, remarking that its proof is more or less identical to that of Lemma \ref{cmff}.
 
\begin{thr} \label{CMtheorem}(Quiver equivariant center manifold theorem) 
Let $({\bf E}, {\bf R})$ be a representation of a quiver ${\bf Q}=\{A\rightrightarrows^s_t V\}$ and let ${\bf L} \in {\rm End}({\bf E}, {\bf R})$ and  ${\bf F}\in C^{k}({\bf E}, {\bf R})$ ($k=1,2,\ldots$) with ${\bf F}(0)=0$ and $D{\bf F}(0)={\bf L}$. 

So we assume that for every $v\in V$ there is a linear map $L_v:E_v\to E_v$ and a $C^k$ smooth  map $F_v:E_v \to E_v$ with $F_v(0)=0$, $DF_v(0)=L_v$, such that
$$R_a \circ L_{s(a)} = L_{t(a)} \circ R_a \ \mbox{and}\  R_a \circ F_{s(a)} = F_{t(a)} \circ R_a    \ \mbox{for all} \ a\in A\, .$$ 
Assume moreover that each of the $L_v$ and $F_v$ ($v\in V$) satisfy the  bounds of Theorem \ref{cmthm}, so that each $F_v$ admits a unique global center manifold $M^{c}_v$.

Then $R_a$ maps the global center manifold of $F_{s(a)}$ into that of $F_{t(a)}$, i.e.,
$$R_a(M^c_{s(a)}) \subset M^c_{t(a)}\, .$$
Moreover, 
the center manifold reductions $F_v^{ c}: E_v^{ c} \to E_v^{ c}$ ($v\in V$) satisfy 
$$R_a \circ F_{s(a)}^{ c} =   F_{t(a)}^{ c} \circ R_a  \ \mbox{for all} \ a\in A\, .$$ 
So the $F^c_v$ define a {\bf Q}-equivariant vector field ${\bf F}^{c}$ on the subrepresentation ${\bf E}^{ c}$ of $({\bf E}, {\bf R})$ consisting of the center subspaces $E^{ c}_v$ ($v\in V$).
\end{thr} 
\begin{proof}
Fix an $a\in A$. By Proposition \ref{eigenspacesprop} we have that $R_a(E_{s(a)}^{ c}) \subset E_{t(a)}^{ c}$ and  $R_a(E_{s(a)}^{ h}) \subset E_{t(a)}^{ h}$, so in particular it holds that 
$$R_a \circ \pi_{s(a)}^{ c} = \pi_{t(a)}^{ c} \circ R_a \ \mbox{and}\ R_a \circ \pi_{s(a)}^{ h} = \pi_{t(a)}^{ h} \circ R_a\, .$$ 
Next, choose an $x\in M^{ c}_{s(a)}$ and recall that for such $x$ we have 
$$\sup_{t\in \mathbb{R}} || (\pi^{ h}_{s(a)} \circ e^{t F_{s(a)}}) (x)|| < \infty\, .$$
Because  $R_{a} \circ e^{t F_{s(a)}} =  e^{t F_{t(a)}} \circ R_a$ and $R_a \circ \pi_{s(a)}^{ h} = \pi_{t(a)}^{ h} \circ R_a$, this implies that
\begin{align}\nonumber
\sup_{t\in \mathbb{R}} || (\pi^{ h}_{t(a)}  \circ e^{t F_{t(a)}}) (R_a(x)) || & = \sup_{t\in \mathbb{R}} || R_a ( \pi^{ h}_{s(a)}  \circ e^{t F_{s(a)}}) (x)) ||  \\ \nonumber
 \leq ||R_a||  \cdot \sup_{t\in \mathbb{R}}|| ( \pi^{ h}_{s(a)}&  \circ e^{t F_{s(a)}})(x) || < \infty \, ,
 \end{align}
where $||R_a||$ is the operator norm of $R_a$. We conclude that $R_a (x) \in M^{ c}_{t(a)}$. This proves that 
$$R_a (M^{ c}_{s(a)}) \subset M^{ c}_{t(a)}\, .$$
Next, recall that if $x\in M^{ c}_{s(a)}$,  then it is of the form 
\begin{align} \nonumber 
& x  = \underbrace{x^c}_{\in E^{ c}_{s(a)}} +  \underbrace{\phi_{s(a)}(x^c)}_{\in E^{ h}_{s(a)}} \, ,
\end{align}
where $\phi_{s(a)}: E^{ c}_{s(a)} \to E^{ h}_{s(a)}$ is the $C^k$ function whose graph is $M_{s(a)}^{ c}$. 
Applying $R_a$ to this equality we find that 
$$R_a (x) = \underbrace{R_a (x^c)}_{\in E_{t(a)}^{ c}} + \underbrace{R_a (\phi_{s(a)}(x^c))}_{\in E^{ h}_{t(a)}} \in M^{ c}_{t(a)}\, .$$
But every $X\in M^{ c}_{t(a)}$ can uniquely be written in  the form 
\begin{align} \nonumber 
& X  = \underbrace{X^c}_{\in E^{ c}_{t(a)}} +  \underbrace{\phi_{t(a)}(X^c)}_{\in E^{ h}_{t(a)}} \, ,
\end{align}
where $\phi_{t(a)}: E^{ c}_{t(a)} \to E^{ h}_{t(a)}$ is the $C^k$ function whose graph is $M^{ c}_{t(a)}$.
This proves that $R_a(\phi_{s(a)}(x^c)) = \phi_{t(a)} (R_a(x^c))$, i.e., that
 $$R_a \circ \phi_{s(a)} = \phi_{t(a)} \circ R_a \, .$$
Recalling the definition of the center manifold reductions $F^c_{v}: E_{v}^c \to E_{v}^c $, we finish by noticing that 
\begin{align}\nonumber 
R_a(F^c_{s(a)} (x^c) ) =&  (R_a\circ \pi^c_{s(a)} \circ F_{s(a)} )(x^c + \phi_{s(a)}( x^c))  \\ \nonumber
 = & (\pi^c_{t(a)} \circ F_{t(a)} \circ R_a)(x^c + \phi_{s(a)}( x^c))  \\ \nonumber
= & (\pi^c_{t(a)} \circ F_{t(a)}) ( R_a (x^c) + R_a(\phi_{s(a)}( x^c)))  \\ \nonumber
= & (\pi^c_{t(a)} \circ F_{t(a)}) ( R_a (x^c) + \phi_{t(a)} ( R_a (x^c)))  \\ \nonumber
= & F^c_{t(a)}(R_a(x^c))\, ,
\end{align}
i.e., $R_a\circ F^c_{s(a)} = F^c_{t(a)} \circ R_a$. This finishes the proof. 
\end{proof}
\begin{remk}\label{localproblem}
Theorem \ref{CMtheorem} is a {\bf Q}-equivariant {\it global} center manifold theorem. Assuming that the first derivatives of the nonlinearities $F_v-L_v$ are globally small, and that their higher derivatives are globally bounded, it guarantees the existence of a globally defined center manifold. The global conditions on the nonlinearities are rather unnatural though, as in practice the nonlinearities will only be small in a neighbourhood of the equilibrium under consideration. The global center manifold theorem is a (very important) step in the proof of a {\it local} center manifold theorem - where the global bounds are not required and a center manifold is guaranteed in a small neighbourhood of the equilibrium. 

 Although it is reasonable to assume that a  local version of Theorem \ref{CMtheorem} holds as well, we were so far unable to prove such a theorem for general {\bf Q}-equivariant systems. The problem arises from the way one usually makes the step from a global to a local center manifold theorem: one replaces the unbounded nonlinearities $F_v-L_v$ by globally bounded nonlinearities, for example by replacing the ODE $\frac{dx}{dt} = F_v(x)$ by the ODE $\frac{dx}{dt} = \widetilde F_v(x):= L_vx + \zeta_v(x) (F_v(x)-L_vx)$, where $\zeta_v: E_v \to \mathbb{R}$  is a smooth bump function with $\zeta_v(x) = 1$ for small $||x||$. By shrinking the support of $\zeta_v$ one can then satisfy the assumptions of Theorem \ref{cmthm}. The problem that we encounter is that in general it is unclear how to choose the bump functions $\zeta_v \, (v\in V)$ in such a way that ${\bf Q}$-equivariance is preserved.

This problem can sometimes be circumvented if the ${\bf Q}$-equivariant vector field happens to be an admissible vector field $F^{\bf N}$ for some network ${\bf N}$. In that case one can multiply the nonlinear parts of each of the separate components $F^{\bf N}_n$  of the vector field with a bump function, choosing the same bump function for nodes with the same colour (more precisely, choosing bump functions that are invariant under the symmetry groupoid $\mathbb{G}_{\bf N}$). The resulting vector field $\widetilde F^{\bf N}$ will then have the same network structure as $F^{\bf N}$, and will hence admit for example the same quiver of subnetworks and quiver of quotient networks. 
In \cite{CMR}, \cite{CMRSIREV} it was shown in detail how this works out for so-called {\it fully homogeneous networks with asymmetric inputs}.  It is not hard to see that the same procedure can be applied to the admissible maps of any network, see Definition \ref{defadmissible}.

On the other hand, quiver symmetry is not always the same as network structure. Therefore even proving an equivariant local center manifold theorem for specific quivers remains problematic.  The mentioned fully homogeneous networks with asymmetric inputs are an exception, as we proved in \cite{RinkSanders3}  that such networks admit a quiver symmetry that is equivalent to a particular network structure (which may be more general than the original network structure though). Such a result will not hold for other types of networks and quivers. For instance, it is not clear to us that equivariance of ${\bf F}$ under {\bf QuoQ(N)} implies that ${\bf F}$ is an admissible vector field for some network that is somehow related to {\bf N}. We therefore  do not  know at this moment how to prove a {\bf QuoQ(N)}-equivariant local center manifold theorem.
\end{remk}

\section{Normal forms and quivers}\label{sec9}
The normal form of a local dynamical system displays the system in a ``standard'' or ``simple'' form. Normal forms are an important tool in the study of the dynamics and bifurcations of maps and vector fields near equilibria, cf. \cite{murdock}, \cite{sanvermur}. The goal of this section is to prove  Theorem \ref{normalformtheorem} below. This theorem states that it can be arranged that the normal form of a dynamical system possesses the same quiver symmetry as the original system. For simplicity, we do not consider parameter dependent vector fields in this section (but it is straightforward to prove the same result for systems with parameters as well).

We start by  recalling one of the classical results of normal form theory in Theorem \ref{normalnormalformtheorem}. To this end, let us consider a smooth ODE 
$$\frac{dx}{dt} = F(x) = F^0(x) +  F^1(x)+ F^2(x)+ \ldots$$
on a  finite-dimensional vector space $E$. That is, $F\in C^{\infty}(E)$ is a smooth vector field on $E$, $F(0)=0$, and $F^k \in P^k(E)$ 
where
$$P^k(E):=\{ F^k: E\to E \ |\ F^k   \ \mbox{is a homogeneous polynomial of degree}\ k+1\}\, .$$
The idea is that we now try to make local coordinate transformations $$x\mapsto y = \Phi(x) = x + \mathcal{O}(||x||^2)$$ that simplify (in one way or another) the higher order terms $F^k  \ (k=1, 2, \ldots)$ of $F$. There are various ways to define such coordinate transformations, and there are various ways to define what it means to ``simplify'' a local ODE.  Theorem \ref{normalnormalformtheorem} states one of the many well-known results.
 
\begin{thr}\label{normalnormalformtheorem}(Normal form theorem)
Let $E$ be a finite dimensional real vector space and let $F \in C^{\infty}(E)$ be a smooth vector field with $F(0)=0$ and Taylor expansion 
$$F = F^0 + F^1 + F^2 + \ldots , \ {\it where}\ F^k \in P^k(E)\, .$$
 Then, for every $1\leq r<\infty$, there exists an analytic diffeomorphism $\Phi$, sending an open neighborhood of $0$ in $E$ to an open neighborhood of $0$ in $E$, so that the coordinate transformation $x\mapsto y = \Phi(x) = x + \mathcal{O}(||x||^2)$ transforms the ODE 
 $$\frac{dx}{dt} = F(x)$$ 
 into an ODE of the form
 $$\frac{dy}{dt} = \overline F(y)$$ 
 with
$$  \overline F = F^0 + \overline F^1 +\overline F^2 + \ldots \ {\rm with}\ \overline F^k \in P^k(E)\, ,$$
while at the same time it holds that
\begin{align}\label{commutator}
e^{t L^S}\circ \overline F^k = \overline F^k \circ e^{t L^S} \ \mbox{for all} \ 1\leq k\leq r\ \mbox{and}\ t\in \mathbb{R}\, .
\end{align}
Here, $L^S=(F^0)^S$ denotes the semisimple part of $F_0$ and $e^{t L^S}$ its (linear) time-$t$ flow.
\end{thr}
\noindent To clarify the statement in this theorem, we will now make a number of definitions and observations. A sketch of the proof of Theorem \ref{normalnormalformtheorem} will be given afterwards.
 First of all, for any two smooth vector fields $F, G \in C^{\infty}(E)$ on $E$ one may define the Lie bracket $[F, G]\in C^{\infty}(E)$ as the vector field
 \begin{align}\label{bracketdef}
 [F, G](x) :=  \left. \frac{d}{dt}\right|_{t=0} \!\!\!\!\!  (e^{tF})_*G (x)= DF(x) \cdot G(x) - DG(x) \cdot F(x)\, .
 \end{align}
 Here, $e^{tF}$ denotes the time-$t$ flow of $F$ (which is defined near each $x\in E$ for some positive time) and $(e^{tF})_*G(x) := De^{tF} \cdot G(e^{-tF}(x))$ is the pushforward of the vector field $G$ by the time-$t$ flow of $F$. 
 We say that $F$ and $G$ commute if $[F,G]=0$, which is equivalent to their flows  $e^{tF}$ and $e^{tG}$ commuting, and equivalent to $F$ being equivariant under the flow  $e^{tG}$, and equivalent to $G$ being equivariant under the flow $e^{tF}$. In particular, (\ref{commutator}) is equivalent to 
 $$[L^S, \overline F_k]=0\ \mbox{for all}\ 1\leq k\leq r\, .$$
 We shall also define, for $F\in C^{\infty}(E)$, the linear operator 
 $${\rm ad}_{F}: C^{\infty}(E) \to C^{\infty}(E)\ \mbox{by}\ {\rm ad}_{F}(G):= [F, G]\, .$$
It follows from (\ref{bracketdef}) that, if $F^k\in P^k(E)$ and $G^l\in P^l(E)$, then $[F^k, G^l]\in P^{k+l}(E)$. In other words, ${\rm ad}_{F^k}: P^l(E) \to P^{k+l}(E)$.

We will also use the following result, that we state here without proof.
 \begin{prop}
 Let $L: E\to E$ be a linear map on a finite dimensional vector space $E$. Recall that there are a unique semisimple linear map $L^S:E \to E$ and a unique nilpotent linear map $L^N:E \to E$ so that $$L=L^S+L^N \ \mbox{and}\ [L^S, L^N]: = L^S L^N - L^N L^S =0\, .$$ $L^S$ is called the {\rm semisimple part} of $L$ and $L^N$ the {\rm nilpotent part} of $L$. 
 
  The  semisimple and nilpotent parts of the restriction ${\rm ad}_{L}: P^k(E) \to P^k(E)$ of ${\rm ad}_L$ to $P^k(E)$ are  then given by
   $$\left({\rm ad}_{L} \right)^S = {\rm ad}_{L^S}\ \mbox{and} \  \left({\rm ad}_{L} \right)^N = {\rm ad}_{L^N}\, .$$  
  \end{prop}
  \begin{cor}\label{corrimker}
   It holds that 
    \begin{itemize}
    \item[{\it i)}] $P^k(E) = {\rm im\, ad}_{L^S} \oplus {\rm ker\, ad}_{L^S}$; 
    \item[{\it ii)}] ${\rm im\, ad}_{L^S} \cap P^k(E)  \subset {\rm im\, ad}_{L} \cap P^k(E) $;
    \item[{\it iii)}] ${\rm ker\, ad}_{L} \cap P^k(E) \subset {\rm ker\, ad}_{L^S} \cap P^k(E)$.
     \item[{\it iv)}] ${\rm ad}_L: {\rm im\, ad}_{L^S}\cap P^k(E) \to {\rm im\, ad}_{L^S} \cap P^k(E)$ is an isomorphism.
    \end{itemize}
    \end{cor}

\begin{proof}
For any linear map $M$ on a finite dimensional real vector space $V$ it holds that $V= {\rm im} \, M^S  \oplus {\rm ker} \, M^S$, ${\rm im} \, M^S \subset {\rm im} \, M$ and ${\rm ker} \, M \subset {\rm ker} \, M^S$.  So these identities hold in particular for $M= {\rm ad}_{L}$ and $V=P^k(E)$. 

To prove point {\it iv)}, note that $M(M^S(x))=M^S(M(x))$ so $M$ sends ${\rm im}\, M^S$ into itself. Because ${\rm ker}\, M \subset {\rm ker} \, M^S$, we have that ${\rm ker}\, M \cap {\rm im}\, M^S = \{0\}$.
\end{proof}
 \begin{proof}{\bf [of Theorem \ref{normalnormalformtheorem}] } 
We sketch  the well-known construction of the normal form by means of ``Lie transformations'', providing only those details that are necessary to prove Theorem \ref{normalformtheorem} below, and leaving out any analytical estimates.

First of all, recall that for any smooth vector field $G\in C^{\infty}(E)$ satisfying $G(0)=0$, the time-$t$ flow $e^{tG}$ defines a diffeomorphism of an open neighborhood of $0$ in $E$ to another open neighborhood of $0$ in $E$. Thus we can consider, for any other smooth vector field $F \in C^{\infty}(E)$, the curve $t \mapsto (e^{tG})_*F$ of transformed vector fields. This curve satisfies the initial condition
$$
(e^{0G})_*F =F\, 
$$
together with the linear differential equation
\begin{align} \nonumber 
\frac{d}{dt} (e^{tG})_*F = &  \left. \frac{d}{dh}\right|_{h=0} \!\!\!\!\! (e^{hG})_* ((e^{tG})_*F) = [G, (e^{tG})_*F ] \\ 
& = {\rm ad}_{G}((e^{tG})_*F)\, .
 \end{align}
The second equality holds by definition of the Lie bracket. 
We conclude that  
$$(e^{G})_*F = e^{{\rm ad}_{G}}(F) = F + [G, F]+\frac{1}{2}[G, [G, F]] + \ldots \ .
$$
The diffeomorphism $\Phi$ in the statement of the theorem is now constructed as the composition of a sequence of time-$1$ flows $e^{G^k}$ $(1 \leq k \leq r)$ with $G^k\in  P^k(E)$. We first take $G^1 \in P^1(E)$, so that $F$ is transformed by $e^{G^1}$ into 
$$(e^{G^1})_*F=e^{{\rm ad}_{G^1}}(F) =F^0 + F^{1,1}+ F^{2,1} + \ldots $$
in which 
\begin{align}\nonumber
\begin{array}{ll} F^{1,1} = F^1+[G^1, F^0] & \in P^1(E) \\
F^{2,1} =F^2 + [G^1, F^1] + \frac{1}{2}[G^1, [G^1, F^0]] & \in P^2(E)\\
F^{3,1}=F^3+\ldots & \in P^3(E)\\
\mbox{etc.}& \ 
\end{array}
\end{align}
From now on we shall often use the notation 
$$L:=F^0\, .$$
The reason is that the operator ${\rm ad}_L$ plays an important role in the rest of this proof.
In fact, the idea is that we now try to choose a $G^1\in P^1(E)$ so that 
$$F^{1,1}=  F^{1} + [G^1, F^0] = F^{1} + [G^1, L] = F^{1} - {\rm ad}_{L}(G^1)$$
is as simple as possible. In general, it will not be possible to  arrange that $F^{1,1}$ vanishes completely. But according to Corollary \ref{corrimker} we can write 
$$F^1 = (F^1)^{\rm im} + (F^1)^{\rm ker} \ \mbox{for unique}\ (F^1)^{\rm im} \in {\rm im \ ad}_{L^S} \ {\rm and}\ (F^1)^{\rm ker} \in {\rm ker \ ad}_{L^S} \, .$$ 
Because ${\rm im \ ad}_{L^S} \subset {\rm im\ ad}_L$ (by Corollary \ref{corrimker}), we can then find a $G^1\in P^1(E)$ so that ${\rm ad}_L(G^1)=(F^1)^{\rm im}$. With this choice of $G^1$ it will hold that 
$$F^{1,1} =  (F^1)^{\rm ker} \in {\rm ker \ ad}_{L^S}\, .$$ 
Note that the choice of $G^1$ is not unique. If we replace our $G^1$ by $G^1+H^1$ with $H^1\in {\rm ker\ ad}_L$, then it will still hold that ${\rm ad}_L(G^1)=(F^1)^{\rm im}$ and therefore also $F^{1,1} =  (F^1)^{\rm ker} \in {\rm ker \ ad}_{L^S}$. To remove this freedom in the choice of $G^1$, let us recall from Corollary \ref{corrimker} that ${\rm ad}_{L}: {\rm im \ ad}_{L^S} \to {\rm im \ ad}_{L^S}$ is an isomorphism. Thus there is a unique 
$$G^1\in {\rm im\ ad}_{L^S}\, $$
with the property that ${\rm ad}_L(G^1)=(F^1)^{\rm im}$, and therefore $F^{1,1} \in {\rm ker \ ad}_{L^S}$. It will be important for later that we choose this particular unique $G^1$ to generate our first normalising transformation.

We proceed by picking the unique $G^2 \in {\rm im \ ad}_{L^S} \cap P^2(E)$ that makes that $(e^{{G^2}} \circ e^{{G^1}})_*F= $ $F^0+ F^{1,1} + F^{2,2} + \ldots $ with $F^{2,2} \in {\rm ker\, ad}_{L^S}$. Continuing in this way, after $r$ steps we obtain that 
$$\Phi:=  e^{G^r} \circ \ldots \circ e^{G^1}$$
transforms $F$ into $\Phi_*F=\overline F = F^0 + F^{1,1} + F^{2,2} + \ldots = F^0 + \overline F^1+\overline F^2+ \ldots$ where $\overline F^k \in {\rm ker\, ad}_{L^S}$ for all $1\leq k\leq r$.     

Being the composition of finitely many flows of polynomial vector fields, this $\Phi$ is obviously analytic and well-defined on an open neighbourhood of $0\in E$. 
\end{proof}
\noindent Before we can prove the main result of this section we will make a few technical observations. The first is simply that quiver symmetry is preserved when taking Lie brackets.

 \begin{prop}\label{Lieproperty}
 Let $({\bf E}, {\bf R})$ be a representation of a quiver ${\bf Q}=\{A \rightrightarrows^s_t V\}$ and let ${\bf F}, {\bf G}\in C^{\infty}({\bf E}, {\bf R})$ be smooth equivariant maps. Then also their Lie brackets $[F_v, G_v]$ ($v\in V$) define a smooth equivariant map $[{\bf F}, {\bf G}]\in C^{\infty}({\bf E}, {\bf R})$.

 \end{prop}
\begin{proof} Smoothness of the $[F_v, G_v]$ is clear. If ${\bf F}, {\bf G}\in C^{\infty}({\bf E}, {\bf R})$, then for any arrow $a\in A$ we have that 
 $$R_a \circ F_{s(a)}  = F_{t(a)} \circ R_a \ \mbox{and}\ R_a \circ G_{s(a)}  = G_{t(a)} \circ R_a\, .$$
 Differentiation of the identity $R_a \circ F_{s(a)} = F_{t(a)} \circ R_a$ yields that 
$$R_a \circ DF_{s(a)}  = (DF_{t(a)} \circ R_a)\cdot R_a\, .$$ 
Similarly, $R_a \circ DG_{s(a)}  = (DG_{t(a)} \circ R_a)\cdot R_a$. As a result we obtain that  
\begin{align}\nonumber
R_a\circ [F_{s(a)}, G_{s(a)}] =&  R_a \circ ( DF_{s(a)} \cdot G_{s(a)} - DG_{s(a)}\cdot F_{s(a)} ) \\ \nonumber
 = & (DF_{t(a)} \circ R_a) \cdot (R_a \circ G_{s(a)}) - (DG_{t(a)} \circ R_a)\cdot (R_a \circ F_{s(a)}) \\ \nonumber
 = & (DF_{t(a)} \circ R_a) \cdot (G_{s(a)} \circ R_a  ) - (DG_{t(a)} \circ R_a)\cdot (F_{s(a)}\circ R_a ) \\ \nonumber
 = & ( DF_{t(a)} \cdot G_{t(a)} - DG_{t(a)}\cdot F_{t(a)} ) \circ R_a =  [F_{t(a)}, G_{t(a)}]  \circ R_a \, .
\end{align}
 \end{proof}
\noindent The second technical result states that $S-N$-decomposition respects quiver symmetry.



\begin{prop}
Let $({\bf E}, {\bf R})$ be a representation of a quiver ${\bf Q}=\{A \rightrightarrows^s_t V\}$ and  $\bf L\in \rm{End}(\bf E, \bf R)$. Define the semisimple part ${\bf L}^S$ and the nilpotent part ${\bf L}^N$ of $\bf L$ to consist of the maps $(L_v)^S: E_v \to E_v$ and $(L_v)^N: E_v\to E_v$ (for each $v\in V$) respectively. Then ${\bf L}^S, {\bf L}^N \in \rm{End}(\bf E, \bf R)$.
\end{prop}
\begin{proof}
For any arrow $a\in A$ consider the linear map 
$$L_{(a)} = \left( \begin{array}{cc} L_{s(a)} & 0 \\ 0 & L_{t(a)} \end{array} \right) \ \mbox{from} \ E_{s(a)} \times E_{t(a)} \ \mbox{to}\ E_{s(a)} \times E_{t(a)} \, .$$
To compute the semisimple and nilpotent parts of this $L_{(a)}$, recall that $L_{(a)}^S$ and $L_{(a)}^N$ are polynomial functions of $L_{(a)}$. This means that there is a polynomial 
$p(x) = a_0 + a_1 x + \ldots + a_n x^n$, with $a_0, \ldots, a_n \in \mathbb{R}$, so that $L_{(a)}^S = p(L_{(a)}) = a_0 \mbox{Id} + a_1 L_{(a)} + \ldots + a_nL_{(a)}^n$ and $L_{(a)}^N = (1-p)(L_{(a)})$. 
 Clearly, $$L_{(a)}^S = p(L_{(a)}) = \left( \begin{array}{cc} p(L_{s(a)}) & 0 \\ 0 & p(L_{t(a)}) \end{array} \right)\, .$$
From this it is clear that $p(L_{s(a)})$ must be the semisimple part of $L_{s(a)}$ and $(1-p)(L_{s(a)})$ must be the nilpotent part of $L_{s(a)}$, and similarly for $L_{t(a)}$. For example, because $p(L_{(a)})$ is semisimple it follows that  $p(L_{s(a)})$ and $p(L_{t(a)})$ must both  be semisimple as well. The remaining conditions for an $S-N$-decomposition are checked similarly. 

Finally, recall that $R_a \circ L_{s(a)} = L_{t(a)} \circ R_a$ because $\bf L\in \rm{End}(\bf E, \bf R)$. Hence it follows that $R_a \circ L_{s(a)}^S = R_a \circ p(L_{s(a)}) = p(L_{t(a)}) \circ R_a = L_{t(a)}^S \circ R_a$, and $R_a \circ L_{s(a)}^N = R_a \circ (1-p)(L_{s(a)}) = (1-p)(L_{t(a)}) \circ R_a = L_{t(a)}^N \circ R_a$. 
\end{proof}
\noindent We are now ready for the main result of this section:
\begin{thr}[Quiver equivariant normal form theorem] \label{normalformtheorem}
Let $({\bf E}, {\bf R})$ be a representation of a quiver ${\bf Q}=\{A \rightrightarrows^s_t V\}$. Let ${\bf F}\in C^{\infty}({\bf E}, {\bf R})$ be a smooth ${\bf Q}$-equivariant vector field, i.e., it consists of vector fields 
$$F_v:E_v\to E_v \ (v\in V)\ \mbox{satisfying}\ F_{t(a)}\circ R_a = R_a \circ F_{s(a)}\ (a\in A)\, .$$ 
We assume that ${\bf F}(0)=0$, i.e., $F_v(0)=0$ for all $v\in V$, and we write  
$${\bf F}= {\bf F}^0+ {\bf F}^1+ {\bf F}^2+ \ldots  \ \mbox{with}\ {\bf F}^k \in P^k({\bf E}, {\bf R}) \, ,$$ 
i.e., $F_v^k\in P^k(E_v)$ for each $v\in V$ and $F_{t(a)}^k\circ R_a = R_a \circ F^k_{s(a)}$ for all $a\in A$. 

Then the local normal forms $\overline{F}_v = F_v^0+\overline F_v^1 + \overline F_v^2 + \ldots$ ($v\in V$) of  the vector fields $F_v=F_v^0+  F_v^1 +   F_v^2 + \ldots$  constructed in Theorem \ref{normalnormalformtheorem} satisfy
$$ R_a \circ \overline{F}_{s(a)} = \overline{F}_{t(a)} \circ R_a\ \mbox{for each}\ a\in A\, .$$
So they define a smooth ${\bf Q}$-equivariant vector field $\overline {\bf F}$ and polynomial ${\bf Q}$-equivariant vector fields $\overline {\bf F}^k$ 
 on an open neighbourhood of $0$ in $({\bf E}, {\bf R})$.
\end{thr}
\begin{proof}
Recall from the proof of Theorem \ref{normalnormalformtheorem} that each of the vector fields
$$F_v = F_v^0 + F_v^1+ F_v^2 + \ldots \in C^{\infty}(E_v)$$ 
is brought into normal form by a sequence of transformations $e^{G_v^k}$. We will show  that  the generators $G_v^k$ of these transformations satisfy $G_{t(a)}^k \circ R_a = R_a \circ G_{s(a)}^k$ for every  $a\in A$. From this it follows that $R_a \circ e^{G_{s(a)}^k} = e^{G_{t(a)}^k}  \circ R_a$ and hence that $R_a \circ \overline F_{s(a)} = \overline F_{t(a)} \circ R_a$. 
The proof is by induction on $k$. 

So let us assume that 
$$R_a \circ G_{s(a)}^j = G_{t(a)}^j \circ R_a \ \mbox{for every} \ a\in A\ \mbox{and every}\ j=1, \ldots k-1\, .$$ 
We recall that the generator $G_v^k$ is the unique vector field in  ${\rm im  \ ad}_{L_v^S} \cap P^k(E_v)$ that solves the  equation 
\begin{align} \label{homologicalv}
F_{v}^{k, k-1} - {\rm ad}_{L_{v}}(G_{v}^k) \in {\rm ker \ ad}_{L^S_{v}} \ \mbox{for all}\ v\in V\, .
\end{align}
Here $F_{v}^{k, k-1} \in P^k(E_v)$. Importantly, our induction hypothesis implies that 
$$R_a \circ F_{s(a)}^{k, k-1} = F_{t(a)}^{k, k-1} \circ R_a\ \mbox{for all}\ a\in A\, .$$ 
We will now show that  this implies that $R_a \circ G_{s(a)}^k = G_{t(a)}^k\circ R_a$ for every $a\in A$. 
The proof of this fact is somewhat technical and goes as follows. 

For any arrow $a\in A$, let us  define the space of conjugate pairs of homogeneous polynomial vector fields
$$P^k_{(a)} := \{(F_{s(a)}^k, F_{t(a)}^k) \in P^k(E_{s(a)})\times P^k(E_{t(a)}) \, |\, R_a \circ F_{s(a)}^k  = F_{t(a)}^k \circ R_a \} \, ,$$
\noindent
and for any ${\bf L}\in {\rm End}({\bf E}, {\bf R})$, let us define the linear map ${\rm ad}_{L_{(a)}}: P^k_{(a)}  \to P^k_{(a)}$  
by
$$
{\rm ad}_{L_{(a)}} (F_{s(a)}^k, F_{t(a)}^k) := ( {\rm ad}_{L_{s(a)}}(F^k_{s(a)}), {\rm ad}_{L_{t(a)}}(F^k_{t(a)}))\, . 
$$
Note that this map is well-defined by Proposition \ref{Lieproperty}: it maps  $P^k_{(a)}$ into $P^k_{(a)}$. Note moreover that equation (\ref{homologicalv}) for $v=s(a)$ and equation (\ref{homologicalv}) for $v=t(a)$ together read
$$\underbrace{(F_{s(a)}^{k, k-1}, F_{t(a)}^{k, k-1})}_{\in P^k_{(a)}}  - {\rm ad}_{L_{(a)}}\underbrace{(G_{s(a)}^k, G_{t(a)}^k)}_{{\rm is\, \, this\, \, in}\ P^k_{(a)}{\rm \, ?}} \in {\rm ker \ ad}_{L^S_{(a)}}\, .$$
We dedicate a separate proposition to the following observation.
\begin{prop}\label{SNadad}
The $S-N$-decomposition of ${\rm ad}_{L_{(a)}}: P^k_{(a)} \to P^k_{(a)}$ is given by 
$$
({\rm ad}_{L_{(a)}})^S = {\rm ad}_{L_{(a)}^S} \ \mbox{and} \ ({\rm ad}_{L_{(a)}})^N= {\rm ad}_{L_{(a)}^N}\, .
$$
\end{prop}
\begin{proof}{\bf [of Proposition \ref{SNadad}]}
Note that ${\rm ad}_{L_{(a)}}$ is the restriction of the product  
$${\rm ad}_{L_{s(a)}} \times {\rm ad}_{L_{t(a)}}: P^k(E_{s(a)}) \times P^k(E_{t(a)}) \to P^k(E_{s(a)}) \times P^k(E_{t(a)}) \, $$
to $P^k_{(a)}$. The $S-N$-decomposition of this product map is clearly given by the product of the $S-N$-decompositions, i.e.,
\begin{align}\nonumber
& ({\rm ad}_{L_{s(a)}} \times {\rm ad}_{L_{t(a)}})^S = {\rm ad}_{L_{s(a)}^S} \times {\rm ad}_{L_{t(a)}^S} \, ,  \\  \nonumber
&  ({\rm ad}_{L_{s(a)}} \times {\rm ad}_{L_{t(a)}})^N = {\rm ad}_{L_{s(a)}^N} \times {\rm ad}_{L_{t(a)}^N}\, .
\end{align}
This can be checked directly, by verifying that the right hand sides satisfy the requirements for the $S-N$-decomposition of ${\rm ad}_{L_{s(a)}} \times {\rm ad}_{L_{t(a)}}$.

But the restriction of ${\rm ad}_{L_{s(a)}^S} \times {\rm ad}_{L_{t(a)}^S}$ to  $P^k_{(a)}$ is ${\rm ad}_{L_{(a)}^S}$. And the restriction of ${\rm ad}_{L_{s(a)}^N} \times {\rm ad}_{L_{t(a)}^N}$ to  $P^k_{(a)}$ is ${\rm ad}_{L_{(a)}^N}$. Moreover, ${\rm ad}_{L_{(a)}^S}$ and ${\rm ad}_{L_{(a)}^N}$ leave $P^k_{(a)}$ invariant as ${\bf L}_{(a)}^S, {\bf L}_{(a)}^N \in {\rm End}({\bf E}, {\bf R})$. This proves the proposition, because the $S-N$-decomposition of the restriction is the restriction of the $S-N$-decomposition.
\end{proof}
\noindent We continue the proof of Theorem \ref{normalformtheorem}. Note that Proposition \ref{SNadad} implies that  
$$P^k_{(a)} = {\rm im\ ad}_{L^S_{(a)}} \oplus {\rm ker\ ad}_{L^S_{(a)}} \, .$$
Just like in the proof of Theorem \ref{normalnormalformtheorem} we can thus uniquely decompose
\begin{align}\label{splitF}
\hspace{-.2cm} (F^{k,k-1}_{s(a)}, F^{k,k-1}_{t(a)}) 
 = \underbrace{((\widetilde{F}^{k,k-1}_{s(a)})^{\rm im}, (\widetilde{F}^{k,k-1}_{t(a)})^{\rm im})}_{\in \, {\rm im\, ad}_{L_{(a)}^S} \! \cap \, P^k_{(a)}} + \underbrace{((\widetilde{F}^{k,k-1}_{s(a)})^{\rm ker}, (\widetilde{F}^{k,k-1}_{t(a)})^{\rm ker})}_{\in \, {\rm ker\, ad}_{L_{(a)}^S}\! \cap \, P^k_{(a)}} \, .
\end{align}
By definition of  ${\rm ad}_{L_{(a)}^S}$, equation (\ref{splitF}) just means that  
$$F^{k,k-1}_{s(a)}  = \underbrace{(\widetilde{F}^{k,k-1}_{s(a)})^{\rm im}}_{\in \, {\rm im\, ad}_{L_{s(a)}^S}} + \underbrace{(\widetilde{F}^{k,k-1}_{s(a)})^{\rm ker}}_{\in \, {\rm ker\, ad}_{L_{s(a)}^S}} \ \mbox{and} \ F^{k,k-1}_{t(a)}  = \underbrace{(\widetilde{F}^{k,k-1}_{t(a)})^{\rm im}}_{\in \, {\rm im\, ad}_{L_{t(a)}^S}} + \underbrace{(\widetilde{F}^{k,k-1}_{t(a)})^{\rm ker}}_{\in \, {\rm ker\, ad}_{L_{t(a)}^S}} \, .$$
But we already know  from the proof of Theorem \ref{normalnormalformtheorem} that the latter two decompositions are unique inside $P^k(E_{s(a)})$ and $P^k(E_{t(a)})$ respectively. So we conclude that 
\begin{align}
& \left( (\widetilde{F}_{s(a)}^{k, k-1})^{\rm im},  (\widetilde{F}_{t(a)}^{k, k-1})^{\rm im} \right) =  \left( (F_{s(a)}^{k, k-1})^{\rm im},  (F_{t(a)}^{k, k-1})^{\rm im} \right) \, \mbox{and} 
\nonumber \\ \nonumber 
 & \left( (\widetilde{F}_{s(a)}^{k, k-1})^{\rm ker},  (\widetilde{F}_{t(a)}^{k, k-1})^{\rm ker} \right) =  \left( (F_{s(a)}^{k, k-1})^{\rm ker},  (F_{t(a)}^{k, k-1})^{\rm ker} \right) \, ,
\end{align}
where $(F_{s(a)}^{k, k-1})^{\rm im}$, $(F_{s(a)}^{k, k-1})^{\rm ker}$, $(F_{t(a)}^{k, k-1})^{\rm im}$, $(F_{t(a)}^{k, k-1})^{\rm ker}$ are the unique vector fields given in the proof of Theorem  \ref{normalnormalformtheorem}. 
In particular it holds that $$R_a \circ (F_{s(a)}^{k, k-1})^{\rm im} = (F_{t(a)}^{k, k-1})^{\rm im} \circ R_a\ \mbox{and}\ R_a \circ (F_{s(a)}^{k, k-1})^{\rm ker} = (F_{t(a)}^{k, k-1})^{\rm ker} \circ R_a\, .$$ 
Next, we note that Proposition \ref{SNadad} implies that ${\rm ad}_{L_{(a)}}: {\rm im\, ad}_{L_{(a)}^S} \to {\rm im\, ad}_{L_{(a)}^S}$ is an isomorphism. Hence there is a unique $G^k_{(a)} = (\widetilde{G}^k_{s(a)}, \widetilde{G}^k_{t(a)}) \in {\rm im \ ad}_{L^S_{(a)}} \! \cap \, P^k_{(a)}$ such that 
\begin{align}\label{homoldouble}
{\rm ad}_{L_{(a)}} \underbrace{(\widetilde{G}^k_{s(a)}, \widetilde{G}^k_{t(a)})}_{\in \, {\rm im \, ad}_{L^S_{(a)}} \! \cap \, P^k_{(a)}} = ((F^{k,k-1}_{s(a)})^{\rm im}, (F^{k,k-1}_{t(a)})^{\rm im})\, .
\end{align}
By definition of ${\rm ad}_{L_{(a)}}$ and ${\rm ad}_{L_{(a)}^S}$, equation (\ref{homoldouble}) just means that 
$${\rm ad}_{L_{s(a)}} (\!\!\!\!\!\!\!\!\!\! \underbrace{\widetilde{G}^k_{s(a)}}_{\in \, {\rm im \, ad}_{L^S_{s(a)}} \! \cap \, P^k_{s(a)}}\!\!\!\!\!\!\!\!\!\! ) = (F^{k,k-1}_{s(a)})^{\rm im} \ {\rm and}\ \, {\rm ad}_{L_{t(a)}} (\!\!\!\!\!\!\!\!\!\! \underbrace{\widetilde{G}^k_{t(a)}}_{\in \, {\rm im \, ad}_{L^S_{t(a)}} \! \cap \, P^k_{t(a)}}\!\!\!\!\!\!\!\!\!\! ) = (F^{k,k-1}_{t(a)})^{\rm im} \, .$$
Again, we already know  from the proof of Theorem \ref{normalnormalformtheorem} that the solutions $\widetilde{G}^k_{s(a)}$ and $\widetilde{G}^k_{t(a)}$ to these two equations are unique. We conclude that 
$$\left( \widetilde{G}_{s(a)}^{k}, \widetilde{G}_{s(a)}^{k} \right) =  \left( G_{s(a)}^{k}, G_{s(a)}^{k} \right)$$
where $G_{s(a)}^{k}$ and $G_{t(a)}^{k}$ are  the unique vector fields given in the proof of Theorem  \ref{normalnormalformtheorem}. 
In particular it holds that  
$$R_a \circ G_{s(a)}^k = G_{t(a)}^k \circ R_a\, .$$
This proves that the $G_v^k$ ($v\in V$) define an equivariant vector field ${\bf G}^k\in P^k({\bf E}, {\bf R})$, which concludes the proof of the induction step and hence the proof of the theorem.
 \end{proof}

\section{An example}\label{sec10}
We finish this paper with an example of a network dynamical system admitting a symmetry quiver that does not only consist of subnetworks or quotients. We consider the network {\bf N} in Figure \ref{picchap10}. 

\begin{figure}[h]\renewcommand{\figurename}{\rm \bf \footnotesize Figure} 
\begin{center}
\hspace{0cm}
  \hspace{-.4cm} 
 \begin{tikzpicture}[->, scale=.75]
	 	 \tikzstyle{vertextype2} = [circle, draw, fill=yellow, minimum size=10pt,inner sep=1pt]
		 \tikzstyle{vertextype3} = [rectangle, draw, fill=green, minimum size=10pt,inner sep=1pt]

	 \node[vertextype2] (v1) at (-4,6) {$1$};
	 \node[vertextype3] (v2) at (-2,6) {$2$};
	 \node[vertextype2] (v3) at (0,6){$3$};
	 \node[vertextype3] (v4) at (2,6) {$4$};
	 \node[vertextype2] (v5) at (4,6) {$5$};

	\node at (-5, 6) {${\bf N}$};

	\path[]
	(v2) edge [thick, blue, bend right = 0] node {} (v1)
	(v4) edge [thick, blue, bend right = 40] node {} (v3)
	(v4) edge [thick, blue, bend right = 0] node {} (v5)

	(v3) edge [thick, red, bend right = 0] node {} (v2)
	(v3) edge [thick, red, bend right = 40] node {} (v4)

	;
 \end{tikzpicture} 
 \vspace{-.3cm}
 \caption{\footnotesize {\rm A network with two types of nodes. Self loops corresponding to internal dynamics are not drawn.}}
 \vspace{-.6cm}
\label{picchap10}
\end{center}
 \end{figure}
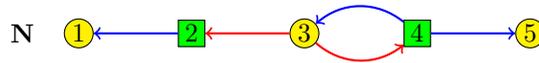 
\noindent Its admissible maps take the general form
\begin{equation}\label{exchap10-1}  
F^{{\bf N}}\left( \begin{array}{c} 
x_1 \\ y_2 \\ x_3 \\ y_4 \\ x_5
\end{array} 
\right) = \left(
 \begin{array}{l}   
    f(x_{1}, \bl{ y_{2}}) \\
 g(y_{2}, \ro{x_{3}}) \\
  f(x_{3}, \bl{ y_{4}}) \\ 
  g(y_{4}, \ro{x_{3}})\\
 f(x_{5}, \bl{ y_{4}}) \\
 \end{array}\right) \, .
 \hspace{-.8cm}
 \end{equation} 
 We will assume that all the variables are one-dimensional, i.e., that $x_1, y_2, x_3$, $y_4, x_5 \in \R$ and $f,g: \R^2 \rightarrow \R$. 
 
 To study this class of maps we will use the $3$-vertex quiver {\bf Q} shown in Figure \ref{picchap102}. The vector spaces corresponding to the vertices of {\bf Q} are given by $E_1 = \R^5$ (for the vertex at ${{\bf N}_1} = {\bf N}$), $E_2 = \R^4$ (for the vertex at ${{\bf N}_2}$) and $E_3 = \R^3$ (for the vertex at ${{\bf N}_3}$). The linear maps of the representation are given by
\begin{align}\label{linmaps10}
\begin{array}{ll} R_{a_1}(x_1, y_2, x_3, y_4, x_5) &= (x_1, y_2, x_3, y_4)\, , \\  
R_{a_2}(x_1, y_2, x_3, y_4, x_5) &= (x_5, y_4, x_3, y_4)\, , \\  
R_{a_3}(y_1, x_2, y_3) &= (x_2, y_3, x_2, y_3)\, , \\  
R_{a_4}(x_1, y_2, x_3, y_4) &= (y_2, x_3, y_4)\, .
\end{array}
\end{align}

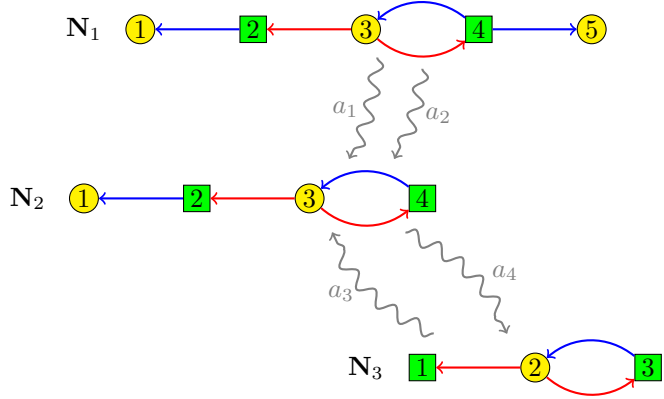
\begin{figure}[h]\renewcommand{\figurename}{\rm \bf \footnotesize Figure} 
\begin{center}
\hspace{0cm}
  \hspace{-.4cm} 
 \begin{tikzpicture}[->, scale=.75]
	 	 \tikzstyle{vertextype2} = [circle, draw, fill=yellow, minimum size=10pt,inner sep=1pt]
		 \tikzstyle{vertextype3} = [rectangle, draw, fill=green, minimum size=10pt,inner sep=1pt]

	 \node[vertextype2] (v1) at (-4,6) {$1$};
	 \node[vertextype3] (v2) at (-2,6) {$2$};
	 \node[vertextype2] (v3) at (0,6){$3$};
	 \node[vertextype3] (v4) at (2,6) {$4$};
	 \node[vertextype2] (v5) at (4,6) {$5$};

	 \node[vertextype2] (v11) at (-5,3) {$1$};
	 \node[vertextype3] (v12) at (-3,3) {$2$};
	 \node[vertextype2] (v13) at (-1,3) {$3$};
	 \node[vertextype3] (v14) at (1,3) {$4$};
	 
	 \node[vertextype3] (v21) at (1, 0) {$1$};
	 \node[vertextype2] (v22) at (3, 0) {$2$};
	 \node[vertextype3] (v23) at (5, 0) {$3$};

	\node at (-5, 6) {${\bf N}_1$};
	\node at (-6, 3) {${\bf N}_2$};
	\node at (0 ,0) {${\bf N}_3$};

	\path[]
	(v2) edge [thick, blue, bend right = 0] node {} (v1)
	(v4) edge [thick, blue, bend right = 40] node {} (v3)
	(v4) edge [thick, blue, bend right = 0] node {} (v5)

	(v12) edge [thick, blue, bend right = 0] node {} (v11)
	(v14) edge [thick, blue, bend right = 40] node {} (v13)
	
	(v23) edge [thick, blue, bend right = 40] node {} (v22)
	(v22) edge [thick, red, bend right = 0] node {} (v21)

	(v3) edge [thick, red, bend right = 0] node {} (v2)
	(v3) edge [thick, red, bend right = 40] node {} (v4)
	
	(v13) edge [thick, red, bend right = 0] node {} (v12)
	(v13) edge [thick, red, bend right = 40] node {} (v14)
	
	(v22) edge [thick, red, bend right = 40] node {} (v23)

	(1, 5.3) edge [thick, decorate, decoration=snake, bend left = 10, gray]  node [right, midway] {\hspace{1pt}$a_{2}$} (0.5, 3.7) 
	(0.2, 5.5) edge [thick, decorate, decoration=snake, gray, bend left = 10]  node [left, midway] {$a_1$} (-0.3, 3.7) 
	(0.7, 2.4) edge [thick, decorate, decoration=snake, bend left = 20, gray]  node [right, midway] {\hspace{6pt}$a_{4}$} (2.5, 0.6) 
	(1.2, 0.6) edge [thick, decorate, decoration=snake, gray, bend left = 20]  node [left, midway] {$a_3$\,\,} (-0.6, 2.4) 

	;
 \end{tikzpicture} 
 \vspace{-.3cm}
 \caption{\footnotesize {\rm A quiver involving the network ${\bf N} = {\bf N}_1$ of Figure \ref{picchap10}.}}
 \vspace{-.4cm}
\label{picchap102}
\end{center}
 \end{figure}  
\noindent We moreover define the quiver equivariant map ${\bf F} = \left\{F^{{\bf N}_1}, F^{{\bf N}_2}, F^{{\bf N}_3}\right\}$, where $F^{{\bf N}_1} = F^{{\bf N}}$ is given by equation \eqref{exchap10-1}, and $F^{{\bf N}_2}$ and $F^{{\bf N}_3}$ are given by 
\begin{equation}\label{exchap10-2}   \nonumber 
F^{{\bf N}_2}\left( \begin{array}{c} 
x_1 \\ y_2 \\ x_3 \\ y_4 
\end{array} 
\right) = \left(
 \begin{array}{l}   
    f(x_{1}, \bl{ y_{2}}) \\
 g(y_{2}, \ro{x_{3}}) \\
  f(x_{3}, \bl{ y_{4}}) \\ 
  g(y_{4}, \ro{x_{3}})\\
 \end{array}\right) \, , \quad
 F^{{\bf N}_3}\left( \begin{array}{c} 
y_1 \\ x_2 \\ y_3 
\end{array} 
\right) = \left(
 \begin{array}{l}   
    g(y_{1}, \ro{ x_{2}}) \\
 f(x_{2}, \bl{y_{3}}) \\
  g(y_{3}, \ro{ x_{2}}) \\ 
 \end{array}\right) \, .
 \hspace{-.8cm}
 \end{equation} 
A direct calculation shows that indeed 
\begin{align}\label{quividentts}
\begin{array}{ll} F^{{\bf N}_2} \circ R_{a_1} = R_{a_1} \circ F^{{\bf N}_1}\, & F^{{\bf N}_2} \circ R_{a_2} = R_{a_2} \circ F^{{\bf N}_1}\, , \\
F^{{\bf N}_2} \circ R_{a_3} = R_{a_3} \circ F^{{\bf N}_3} \, , & F^{{\bf N}_3} \circ R_{a_4} = R_{a_4} \circ F^{{\bf N}_2}\, .
\end{array}
\end{align}
Alternatively, note that $F^{{\bf N}_1}$, $F^{{\bf N}_2}$ and $F^{{\bf N}_3}$ are the admissible maps for the networks ${{\bf N}_1}$, ${{\bf N}_2}$ and ${{\bf N}_3}$ shown in Figure \ref{picchap102}. It can easily be seen that the linear maps $R_{a_1}, R_{a_2}, R_{a_3}, R_{a_4}$ are induced by graph fibrations between the networks, so that the identities \eqref{quividentts} follow from Theorem \ref{devilletheorem1}. 

 The attentive reader might wonder why we have chosen the specific quiver of Figure \ref{picchap102}. After all, this quiver does not contain all subnetworks of ${\bf N}_1$, nor does it contain all of its quotient networks. For instance, the subnetwork of ${\bf N}_1$ consisting of  nodes $3$ and $4$ (which is also a quotient of ${\bf N}_1$) is absent. To better explain our choice of the quiver, we need the following proposition.

\begin{prop}\label{quivequitonet}
Let $({\bf E}, {\bf R})$ be the representation of the quiver {\bf Q} of Figure \ref{picchap102}, consisting of the vector spaces $E_1 = \R^5$, $E_2 = \R^4$ and $E_3 = \R^3$, and the linear maps $R_{a_1}, \ldots, R_{a_4}$ given in \eqref{linmaps10}. 

A triple of maps ${\bf G} = \left\{G^1, G^2, G^3\right\}$ (with $G^i: E_i \to E_i$) is ${\bf Q}$-equivariant if and only if there exist maps $h: \R^4 \rightarrow \R$ and $l: \R^3 \rightarrow \R$ such that
\begin{align}\label{exchap10-3}  
&G^1\left( \! \begin{array}{c} 
x_1 \\ y_2 \\ x_3 \\ y_4 \\ x_5
\end{array}  \!
\right) = \left( \!
 \begin{array}{l}   
    h(x_{1}, y_{2}, x_3, y_4) \\
 l(y_{2}, x_{3}, y_4) \\
  h(x_{3}, y_{4}, x_3, y_4)\\ 
  l(y_{4}, x_{3}, y_4)\\
 h(x_{5}, y_{4}, x_3, y_4)
 \end{array} \! \right) \, , \\ \nonumber
 &G^2\left( \! \begin{array}{c} 
x_1 \\ y_2 \\ x_3 \\ y_4 
\end{array} \! 
\right) = \left( \!
 \begin{array}{l}   
    h(x_{1}, y_{2}, x_3, y_4) \\
 l(y_{2}, x_{3}, y_4) \\
  h(x_{3}, y_{4}, x_3, y_4)\\ 
  l(y_{4}, x_{3}, y_4)
 \end{array} \! \right) \ \mbox{and}\ \
 G^3\left( \! \begin{array}{c} 
y_1 \\ x_2 \\ y_3  
\end{array}  \!
\right) = \left( \!
 \begin{array}{l}   
 l(y_{1}, x_{2}, y_3) \\
  h(x_{2}, y_{3}, x_2, y_3)\\ 
  l(y_{3}, x_{2}, y_3)
 \end{array} \! \right) \, .
 \end{align} 
\end{prop}

\begin{proof}
A direct calculation shows that ${\bf G} = \left\{G^1, G^2, G^3\right\}$ as given by equations \eqref{exchap10-3} is indeed ${\bf Q}$-equivariant. In other words, one verifies that $G^2 \circ R_{a_1} = R_{a_1} \circ G^1$, with similar relations for $R_{a_2}$, $R_{a_3}$ and $R_{a_4}$ as in equation \eqref{exchap10-2}. 

Conversely, suppose ${\bf G} = \left\{G^1, G^2, G^3\right\}$ is a ${\bf Q}$-equivariant map. This assumption implies in particular that $G^2 \circ R_{a_3} = R_{a_3}\circ G^3$, where we recall that $R_{a_3}(y_1, x_2, y_3) = (x_2, y_3, x_2, y_3)$. 
Reading off the first component of the identity $(G^2 \circ R_{a_3})(y_1, x_2, y_3)  = (R_{a_3} \circ G^3)(y_1, x_2, y_3)$, we thus see that
\begin{equation}\label{proofthing1}
G^2_1(x_2, y_3, x_2, y_3) = G^3_2(y_1, x_2, y_3)\, ,
\end{equation} 
where we have used that $(R_{a_3}\circ G^3)_1 = G^3_2$. Likewise, evaluating the first component of the identity $G^3 \circ R_{a_4} = R_{a_4} \circ G^2$ gives
\begin{equation}\label{proofthing2}
G^3_1(y_2, x_3, y_4) = G^2_2(x_1, y_2, x_3, y_4)\, .
\end{equation}
Next, we calculate 
\begin{align}\nonumber
&R_{a_4}R_{a_3}(y_1, x_2, y_3) = (y_3, x_2, y_3) \, ,\\  \nonumber
&R_{a_3}R_{a_4}(x_1, y_2, x_3, y_4) = (x_3, y_4, x_3, y_4)\, , \\ \nonumber
&R_{a_4}R_{a_3}R_{a_4}(x_1, y_2, x_3, y_4) = (y_4, x_3, y_4) \, . 
\end{align}
From the identities $G^2 \circ R_{a_3} = R_{a_3} \circ G^3$ and $G^3 \circ R_{a_4} = R_{a_4} \circ G^2$ we get $G^3 \circ R_{a_4} \circ R_{a_3} = R_{a_4} \circ R_{a_3}\circ G^3$, and the first component of this equation reads
\begin{equation}\label{proofthing3}
G^3_1 (y_3, x_2, y_3)=  G^3_3(y_1, x_2, y_3) \, .
\end{equation}
We likewise find $G^2 \circ R_{a_3} \circ R_{a_4} = R_{a_3} \circ R_{a_4}\circ G^2$, so that
\begin{equation}\label{proofthing4}
G^2_1 (x_3, y_4, x_3, y_4) = G^2_3(x_1, y_2, x_3, y_4)\, .
\end{equation}
And finally, using that $G^3 \circ R_{a_4} \circ R_{a_3} \circ  R_{a_4} = R_{a_4} \circ R_{a_3}\circ R_{a_4} \circ G^2$, we have
\begin{align}\label{proofthing5}
G^3_1 (y_4, x_3, y_4)    =  G^2_4(x_1, y_2, x_3, y_4)\, .
\end{align}
If we now set
\begin{align}
h(x_1, y_2, x_3, y_4) := G^2_1(x_1, y_2, x_3, y_4)  \quad \text{ and } \nonumber \quad l(y_1, x_2, y_3) := G^3_1(y_1, x_2, y_3)\, ,
\end{align}
then equations \eqref{proofthing1}, \eqref{proofthing2}, \eqref{proofthing3}, \eqref{proofthing4} and \eqref{proofthing5} show that $G^2$ and $G^3$ have the required form.  

The form of $G^1$ follows from similar arguments involving $R_{a_1}$ and $R_{a_2}$. More precisely, for $i = 1, \dots, 4$, we find 
\begin{equation}
G^2_i \circ R_{a_1}  = (G^2 \circ R_{a_1})_i  = (R_{a_1} \circ G^1)_i  = G^1_i\, .
\end{equation}
Moreover, we have 
\begin{equation}
G^2_1 \circ R_{a_2}  =  (G^2 \circ R_{a_2})_1 = (R_{a_2} \circ G^1)_1 = G^1_5\, .
\end{equation}
This proves that ${\bf G}$ is ${\bf Q}$-equivariant if and only if it is of the form \eqref{exchap10-3}.
\end{proof}
\begin{remk}
Note that each of the maps $G^1$, $G^2$ and $G^3$ in equation \eqref{exchap10-3} may be seen as the admissible map for some network with two types of nodes. For example, $G^1$ is an admissible map for the network ${\bf \widetilde{N}}_1$ shown in Figure \ref{picchap103}. 
Proposition \ref{quivequitonet} therefore shows that a map $\tilde{G}: \R^5 \rightarrow \R^5$ is an admissible map for the network ${\bf \widetilde{N}}_1$, if and only if we have $\tilde{G} = G^1$ for some ${\bf Q}$-equivariant map ${\bf G} = \left\{G^1, G^2, G^3\right\}$ for the quiver ${\bf Q}$  in Figure \ref{picchap102}. 
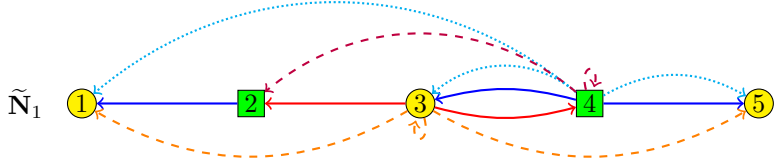
\begin{figure}[h]\renewcommand{\figurename}{\rm \bf \footnotesize Figure} 
\begin{center}
\hspace{0cm}
  \hspace{-.4cm} 
 \begin{tikzpicture}[->, scale=.75]
	 	 \tikzstyle{vertextype2} = [circle, draw, fill=yellow, minimum size=10pt,inner sep=1pt]
		 \tikzstyle{vertextype3} = [rectangle, draw, fill=green, minimum size=10pt,inner sep=1pt]

	 \node[vertextype2] (v1) at (-6,0) {$1$};
	 \node[vertextype3] (v2) at (-3,0) {$2$};
	 \node[vertextype2] (v3) at (0,0){$3$};
	 \node[vertextype3] (v4) at (3,0) {$4$};
	 \node[vertextype2] (v5) at (6,0) {$5$};

	\node at (-7, 0) {${\bf \widetilde{N}}_1$};

	\path[]
	(v2) edge [thick, blue, bend right = 0] node {} (v1)
	(v4) edge [thick, blue, bend right = 15] node {} (v3)
	(v4) edge [thick, blue, bend right = 0] node {} (v5)

	(v3) edge [thick, red, bend right = 0] node {} (v2)
	(v3) edge [thick, red, bend right = 15] node {} (v4)

	(v3) edge [thick, orange, dashed, bend right = -30] node {} (v1)
	(v3) edge [thick, orange, dashed, bend right = 30] node {} (v5)
	(v3) edge [thick, orange, dashed, loop below] node {} (v3)
	
	(v4) edge [thick, cyan, densely dotted, bend right = 40] node {} (v1)
	(v4) edge [thick, cyan, densely dotted, bend right = 40] node {} (v3)
	(v4) edge [thick, cyan, densely dotted, bend right = -30] node {} (v5)
	
	(v4) edge [thick, purple, dashed, bend right = 40] node {} (v2)
	(v4) edge [thick, purple, dashed, loop above] node {} (v4)

	;
 \end{tikzpicture} 
 \vspace{-.3cm}
 \caption{\footnotesize {\rm The network ${\bf \widetilde{N}}_1$ obtained by adding several arrow types to ${\bf {N}}_1$. Self loops corresponding to internal dynamics are not drawn.}}
 \vspace{-.4cm}
\label{picchap103}
\end{center}
 \end{figure} \\
\noindent Comparing the networks ${\bf \widetilde{N}}_1$ and ${\bf {N}}_1$, we see that ${\bf \widetilde{N}}_1$ can be obtained from ${\bf {N}}_1$ by adding arrow types. More precisely, these additional arrows are formed by concatenating two or more existing arrow types (i.e., the red and blue arrows of network ${\bf {N}}_1$). It can be shown that adding concatenations of arrow types in this fashion has no effect on the presence of sub- and quotient networks, cf. Example \ref{interestingffex}. As the network structure of $G^1$ is a consequence of quiver symmetry, we see that all information about sub- and quotient networks in $F^{{\bf N}_1}$ is ``encoded'' in the quiver of Figure \ref{picchap102}. To  obtain such a useful quiver representation, one uses the theory of fundamental networks. We will not explain this in further detail here, but see for instance Section 11 of \cite{fibr}.
\end{remk}

\noindent 
In what follows, we shall determine the properties of generic one-parameter steady state bifurcations for the ODE $\frac{dx}{dt} = F^{{\bf N}}(x; \lambda)$ by means of Lyapunov-Schmidt reduction,  while exploiting the quiver symmetry that is present in the problem. To this end, we let $F^{{\bf N}} = F^{{\bf N}_1}$, $F^{{\bf N}_2}$ and $F^{{\bf N}_3}$ depend on a parameter $\la$, taking values in some open neighbourhood $\Lambda$ of $0 \in \R$. To keep the network structures intact for all values of $\la$, we simply replace $f(x, y)$ and $g(y,x)$ in $F^{{\bf N}_1}$, $F^{{\bf N}_2}$  and $F^{{\bf N}_3}$ by $f(x, y; \la)$ and $g(y, x; \la)$. For instance, we get
\begin{equation}\label{exchap10-4}  
F^{{\bf N}_1}\left( \begin{array}{c} 
x_1 \\ y_2 \\ x_3 \\ y_4 \\ x_5 \\ \la
\end{array} 
\right) = \left(
 \begin{array}{l}   
    f(x_{1}, \bl{ y_{2}}; \la) \\
 g(y_{2}, \ro{x_{3}}; \la) \\
  f(x_{3}, \bl{ y_{4}}; \la) \\ 
  g(y_{4}, \ro{x_{3}}; \la)\\
 f(x_{5}, \bl{ y_{4}}; \la) \\
 \end{array}\right) \, .
 \hspace{-.8cm}
 \end{equation} 
We will assume that $F^{{\bf N}_1}(0; 0) = 0$, from which it follows that also $F^{{\bf N}_2}(0;0)  = F^{{\bf N}_3}( 0; 0) = 0$. If we now define 
\begin{align}
\begin{array}{ll}
a := \left.\frac{\partial f(x,y; \la)}{\partial x}\right|_{(0;0)}  \, &c := \left.\frac{\partial f(x,y; \la)}{\partial y}\right|_{(0;0)}  \, ,\\ 
b := \left.\frac{\partial g(y,x; \la)}{\partial y}\right|_{(0;0)}  \, & d := \left.\frac{\partial g(y,x; \la)}{\partial x}\right|_{(0;0)}  \, ,
\end{array}
\end{align}
then the Jacobian matrices of the network maps (in the direction of the variables $x_i$ and $y_j$, but not $\la$) are given by
\begin{align}
DF^{{\bf N}_1}(0; 0) &=  \begin{pmatrix*}[l]
 a & c & 0 & 0 & 0 \\
 0 & b & d & 0 & 0 \\
 0 & 0 & a & c & 0 \\
 0 & 0 & d & b & 0 \\
0 & 0 & 0 & c & a \\
 \end{pmatrix*} \, ,\\ \nonumber
 DF^{{\bf N}_2}(0; 0) &=  \begin{pmatrix*}[l]
 a & c & 0 & 0  \\
 0 & b & d & 0  \\
 0 & 0 & a & c  \\
 0 & 0 & d & b  \\
 \end{pmatrix*} \text{ and } \,
 DF^{{\bf N}_3}(0; 0) =  \begin{pmatrix*}[l]
 b & d & 0  \\
 0 & a & c  \\
 0 & d & b  \\
 \end{pmatrix*} \, .
 \end{align}
It is not hard to see that $DF^{{\bf N}_1}(0; 0)$ is non-invertible precisely when either $a = 0$, $b= 0$, or 
\begin{equation}\label{matrixofexam}
\det \begin{pmatrix*}[l]
 a & c  \\
 d & b  \\
 \end{pmatrix*}  = ab-cd = 0\, .
\end{equation}
Note that generically only one of these three conditions ($a=0$, $b=0$ or $ab-cd = 0$) is satisfied. We shall study these cases separately. 
\\ \mbox{} \\
\noindent {\bf Case 1:} We start with the case $a = 0$, where we assume in addition that $b \not= 0$ and $ab-cd \not= 0$. It follows that the kernels (which in this case are also the generalised kernels) of $D_XF^{{\bf N}_i}(0; 0)$, $i = 1,2,3$, are given by
\begin{align}\label{kersnelsof1}\nonumber
&\ker(D_XF^{{\bf N}_1}(0; 0)) = \{(x_1, 0,0,0,x_5) \mid x_1, x_5 \in \R\} \subset \R^5\, , \\ 
&\ker(D_XF^{{\bf N}_2}(0; 0)) = \{(x_1, 0,0,0) \mid x_1 \in \R\}  \subset \R^4\, , \\ \nonumber
&\ker(D_XF^{{\bf N}_3}(0; 0)) = \{0\} \subset \R^3 \, .
\end{align}
We will identify these spaces with $\R^2$, $\R$ and $\{0\}$ respectively, using the variables $x_1$ and $x_5$. Recall now that these kernels together form a subrepresentation (it  turns out in fact that this subrepresentation is  {\it indecomposable}). Under our identifications, the quiver symmetries restrict to this subrepresentation as
\begin{align}
&R_{a_1}(x_1, x_5) = x_1 \, , R_{a_2}(x_1, x_5) = x_5 \, , R_{a_3}(0) = 0 \text{ and } R_{a_4}(x_1) = 0\, .
\end{align}
One verifies that a general equivariant map on this subrepresentation must be of the form ${\bf H} = \left\{ H^1, H^2, 0 \right\}$, with 
$$H^1(x_1, x_5) = (h(x_1), h(x_5))\ \mbox{and}\ H^2(x_1) = h(x_1)\, ,$$ 
where $h: \R \rightarrow \R$ is any smooth function satisfying $h(0) = 0$. It thus follows from Theorem \ref{LStheorem} that, after performing equivariant Lyapunov-Smith reduction, the bifurcation equation that needs to be solved for $F^{{\bf N}_1}$ is of the form
$$H^1(x_1, x_5; \la) = (h(x_{1}; \la), h(x_{5}; \la)) = (0,0)\, ,$$ 
where $h(0;\la) = 0$ for all $\la$ close to zero. Interestingly, quiver-equivariance therefore implies that the two-dimensional bifurcation equation decouples into two one-dimensional equations. For each of the two components of $H^1$, we generically find a transcritical bifurcation with one branch satisfying $x_i(\la) = 0$, and the other  satisfying $x_i(\la) \sim \la$ (with $i = 1$ or $i = 5$). We thus get a remarkable {\bf double transcritical bifurcation} in which a total of four bifurcation branches coalesce. Their asymptotics are given by
$$(x_1(\la), x_5(\la)) \sim (0,0), (0,\la), (\la, 0)\ \mbox{and} \ (\la, \la)\, .$$ 
From the description of $\ker(D_XF^{{\bf N}_1}(0; 0))$ in equation \eqref{kersnelsof1} we see that these branches lie in the synchrony spaces $\{x_1 = x_3 = x_5, y_2 = y_4\}$, $\{x_1 = x_3, y_2 = y_4\}$, $\{x_3 = x_5, y_2 = y_4\}$ and $\{x_1 = x_5, y_2 = y_4\}$, in order of listing.\\
\ \mbox{} \\
\noindent {\bf Case 2:}
Next, we investigate what happens in case $b=0$, while $a, ab-cd \not= 0$. It follows that the (generalised) kernels of $D_XF^{{\bf N}_i}(0; 0)$, $i = 1,2,3$, are given by
\begin{align}\label{kersnelsof2}\nonumber
&\ker(D_XF^{{\bf N}_1}(0; 0)) = \{(-ca^{-1}x, x,0,0,0) \mid x \in \R\} \subset \R^5\, ,\\ 
&\ker(D_XF^{{\bf N}_2}(0; 0)) = \{(-ca^{-1}x, x,0,0) \mid x \in \R\}  \subset \R^4\, , \\ \nonumber
&\ker(D_XF^{{\bf N}_3}(0; 0)) = \{(x, 0, 0) \mid x \in \R\} \subset \R^3 \, .
\end{align}
If we now use the variable $x$ to identify each of these spaces with $\R$, then the quiver symmetries become 
\begin{align}\label{Rsubrep}
&R_{a_1}(x) = x \, , R_{a_2}(x) = 0 \, , R_{a_3}(x) = 0 \text{ and } R_{a_4}(x) = x\, .
\end{align}
Again, this defines an indecomposable subrepresentation, non-isomorphic to the one we found for Case 1. It follows from (\ref{Rsubrep}) that an equivariant map now must be of the form ${\bf H} = \left\{ h,h,h\right\}$, with $h: \mathbb{R}\to\mathbb{R}$ and $h(0) = 0$. For $F^{{\bf N}_1}$ we therefore get a single {\bf transcritical bifurcation} with branches $x(\la) = 0$ and $x(\la) \sim \la$. Comparing to the expression for $\ker(D_XF^{{\bf N}_1}(0; 0))$, we see that the former branch lies in the synchrony space $\{x_1 = x_3 = x_5, y_2 = y_4\}$, whereas the latter lies in $\{x_3 = x_5\}$. \\
\mbox{} \\
\noindent {\bf Case 3:} Finally, we assume $ab-cd = 0$ while $a,b \not= 0$. In addition, we make the generic assumption that $a+b \not= 0$. As $a+b$ is the trace of the matrix in equation \eqref{matrixofexam}, this means $D_XF^{{\bf N}_1}(0; 0)$ has a simple eigenvalue $0$. We find that there exists a unique non-zero vector $(s,t) \in \R^2$ such that 
\begin{align}\label{kersnelsof2}\nonumber
&\ker(D_XF^{{\bf N}_1}(0; 0)) = \{x\cdot (s,t,s,t,s) \mid x \in \R\} \subset \R^5\, , \\ 
&\ker(D_XF^{{\bf N}_2}(0; 0)) = \{x\cdot(s,t,s,t) \mid x \in \R\}  \subset \R^4\, , \\ \nonumber
&\ker(D_XF^{{\bf N}_3}(0; 0)) = \{x\cdot(t,s,t) \mid x \in \R\} \subset \R^3 \, .
\end{align}
If we use $x$ to identify these spaces with $\R$, then each map $R_{a_i}$ restricts to the identity. This means that we found yet another non-isomorphic indecomposable subrepresentation. It also means that equivariance poses no restrictions on the reduced maps, and we find a {\bf saddle-node bifurcation} within the maximally synchronous space $\{x_1 = x_3 = x_5, y_2 = y_4\}$. \\
\mbox{} \\
\noindent  Note that triples of admissible maps ${\bf F} = \left\{ F^{\bf N}_1, F^{\bf N}_2, F^{\bf N_3}\right\}$ for the networks ${\bf N}_1, {\bf N}_2$ and ${\bf N}_3$ constitute only a subset of the collection of all  ${\bf Q}$-equivariant maps ${\bf G}=\left\{ G^1, G^2, G^3 \right\}$; see Proposition \ref{quivequitonet}. As a result, we cannot rule out restrictions on the Taylor coefficients of the Lyapunov-Schmidt reduction ${\bf H}$ of ${\bf F}$, in addition to the ones that we found in the above ``generic'' bifurcation analysis. Such additional restrictions could make for different generic bifurcation scenarios from the ones we described, so that a more in depth analysis  is necessary to obtain the generic bifurcations for the admissible maps ${\bf F}$. 

We were indeed able to verify that the bifurcations that we described above for generic reduced quiver equivariant vector fields, are also generic for admissible vector fields ${\bf F}$. We actually found that all one-parameter bifurcation scenarios that are generic for quiver equivariant vector fields ${\bf G}=\{G^1, G^2, G^3\}$, are also generic for admissible vector fields ${\bf F} = \{ F^{\bf N}_1, F^{\bf N}_2, F^{\bf N_3}\}$. As these results are to be expected, we will not prove them here. 

\begin{remk}
We end this paper with a remark on representation theory. Recall that, for each of the three cases of the above example (namely, $a=0$, $b=0$ and $ab-cd=0$), we claimed  that the kernels of the Jacobian matrices form  (non-isomorphic) {\it indecomposable} subrepresentations. In fact, this holds because every endomorphism of the subrepresentation is a scalar multiple of the identity endomorphism (the identity endomorphism consists of an identity map at each node of the quiver). This fact can for instance be seen from our description of the equivariant maps for each of the three cases.

The reader familiar with classical representation theory (i.e. pertaining to compact groups) might recognise in this the definition of an \textit{absolutely irreducible subrepresentation}. Indeed, a representation of a compact group is called absolutely irreducible precisely when all of its endomorphisms are given by scalar multiples of the identity, see \cite{golschaef2}. An important result from classical equivariant theory is that a one-parameter steady state bifurcation can generically only occur along a kernel that is an absolutely irreducible representation of the symmetry. This result is especially powerful when combined with an algebraic result known as the \textit{Krull-Schmidt theorem}. The latter theorem states that any (finite dimensional) representation of a group can uniquely be written as the direct sum of a number of irreducible representations. Without going into technical details, these two results together imply that, up to isomorphism, there are only finitely many subrepresentations that one has to consider in a full investigation of possible bifurcation scenarios. 

A version of the Krull-Schmidt theorem exists for quiver symmetries as well, see \cite{Krause}. We aim to show in a follow-up article that a one-parameter steady state bifurcation in a quiver equivariant ODE occurs generically along an \textit{absolutely indecomposable subrepresentation}. This latter notion means that all endomorphisms of the subrepresentation are scalar multiples of the identity, up to nilpotent maps (in our example, the zero map happens to be the only nilpotent map, but this is not true for general quivers.) We will also show more general results pertaining to generic center subspaces, as well as multiple bifurcation parameters. Analogous results have already been shown for systems with a \textit{monoid} symmetry \cite{schwenker} and \cite{transversality}. A monoid is a generalisation of a group (see Example \ref{monoidex}) but only a special case of a quiver. 
\end{remk}

\section{Acknowledgement}
This research is partly financed by the Dutch Research Council (NWO) via Eddie Nijholt's research program ``Designing Network Dynamical Systems through Algebra''. 

Bob Rink is happy to acknowledge the hospitality and financial support of the Sydney Mathematical Research Institute. 
  \bibliography{CoupledNetworks}
\bibliographystyle{amsplain}

\end{document}